\documentclass[10pt]{article}

\usepackage{
            ulem		
           ,soul		
} 

\usepackage[
  paper  = a4paper,
  left   = 1.0in,
  right  = 1.0in,
  top    = 0.8in,
  bottom = 0.8in,
  ]{geometry}
  
\usepackage[english]{babel}
\usepackage[parfill]{parskip}
\usepackage{afterpage}
\usepackage{enumitem}
\usepackage{framed}
\usepackage{xspace}
\usepackage{float}
\usepackage[nodayofweek]{datetime}

\usepackage[linesnumbered, lined, ruled, commentsnumbered]{algorithm2e}

\usepackage[dvipsnames]{xcolor}
\definecolor{shadecolor}{gray}{0.9}

\usepackage{tikz}
\usetikzlibrary{shapes.geometric, arrows}
\usepackage{atbegshi}

\newcommand{\blue}[1]{\textcolor{MidnightBlue}{#1}}

\usepackage{pifont}
\usepackage{graphicx}
\usepackage{caption}
\usepackage{subfigure}

\usepackage{booktabs,array}

\usepackage[numbers]{natbib}
\usepackage[colorlinks,linktoc=all]{hyperref}
\usepackage[all]{hypcap}
\hypersetup{citecolor=MidnightBlue}
\hypersetup{linkcolor=MidnightBlue}
\hypersetup{urlcolor=MidnightBlue}

\usepackage{tikz}

\usetikzlibrary{shapes.geometric, arrows}
\tikzstyle{startstop} = [rectangle, rounded corners, minimum width=2.5cm, minimum height=0.8cm, text centered, draw=black]
\tikzstyle{process} = [rectangle, minimum width=5cm, minimum height=0.8cm, text width=6cm, text centered, align=center, draw=black]
\tikzstyle{decision} = [diamond, aspect=3, minimum width=6cm, minimum height=2cm, text centered, draw=black]
\tikzstyle{arrow} = [thick,->,>=stealth]

\usepackage{amsmath,amssymb,amsthm}
\usepackage{bm}
\usepackage[mathscr]{eucal}                                      
\usepackage{mathtools}
\mathtoolsset{showonlyrefs=true}

\theoremstyle{plain}
\newtheorem{theorem}{Theorem}
\numberwithin{theorem}{section}

\newtheorem{proposition}[theorem]{Proposition}

\theoremstyle{definition}

\newtheorem{example}[theorem]{Example}

\newtheorem{remark}[theorem]{Remark}
\newtheorem{assumption}[theorem]{Assumption}
\newtheorem{lemma}[theorem]{Lemma}

\newcommand\Eb{\mathbb{E}}

\newcommand\Pb{\mathbb{P}}

\newcommand\Rb{\mathbb{R}}

\newcommand{\bE}{\mathbb{E}}

\newcommand{\bP}{\mathbb{P}}

\newcommand{\bR}{\mathbb{R}}

\newcommand\Kc{\mathscr{K}}

\newcommand\Lc{\mathscr{L}}

\newcommand{\cF}{\mathcal{F}}

\newcommand{\cS}{\mathcal{S}}

\newcommand{\cY}{\mathcal{Y}}

\newcommand\gammat{\widetilde{\gamma}}
\newcommand\rad{\tau}

\newcommand{\RR}{R}
\newcommand{\h}{{\bf h}}

\newcommand\dd{\text{d}}

\DeclareMathOperator*{\argmin}{arg\,min}
\DeclareMathOperator*{\esssup}{ess\,sup}

\usepackage{pstricks}
\usepackage{rotating}
\usepackage{lscape}

\usepackage{graphicx}
\usepackage{tikz}
\usepackage{atbegshi}

\AtBeginShipoutFirst{%
    \begin{tikzpicture}[remember picture, overlay]
        \node[anchor=north west, xshift=1in, yshift=-1.5cm] at (current page.north west) {%
            \includegraphics[width=4cm]{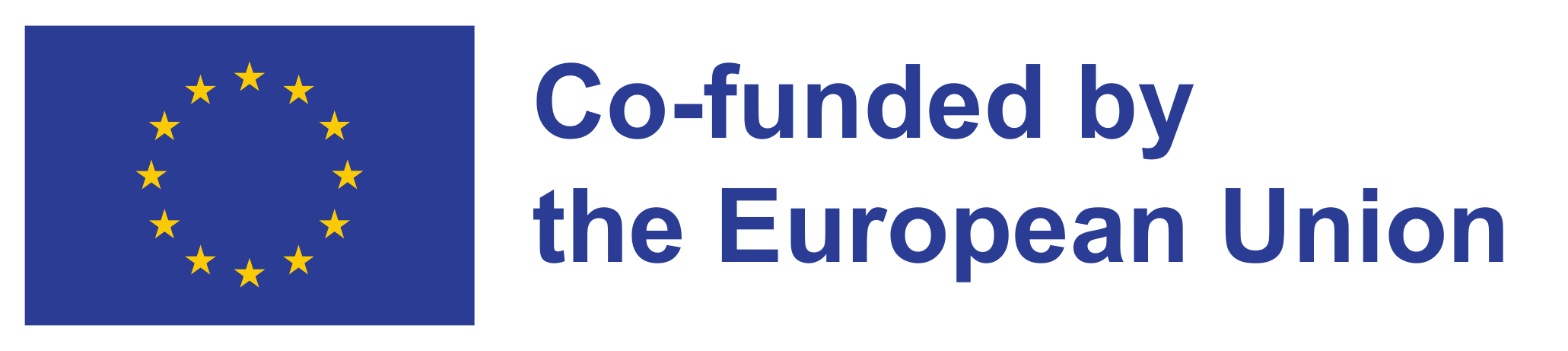} 
        };
    \end{tikzpicture}%
}

\SetKwComment{Comment}{/* }{ */}
\SetKwInOut{KwIn}{Input}
\SetKwInOut{KwData}{Initialization}
\begin{document}

\title{Numerical approximation of McKean-Vlasov SDEs via stochastic gradient descent\footnote{Full code available at \url{https://github.com/amatoandre/Numerical-approximation-of-MKV-SDEs-via-SGD}}}

\author{
\normalsize Ankush Agarwal\textit{$^{a}$} \\
        \small   aagarw93@uwo.ca
\and
\normalsize Andrea Amato\textit{$^{b}$} \\
        \small   andrea.amato9@unibo.it
        \and
\normalsize Stefano Pagliarani\textit{$^{b}$} \\
        \small   stefano.pagliarani9@unibo.it
\and
 \normalsize Gon\c calo dos Reis\textit{$^{c,d}$}\\
        \small  G.dosReis@ed.ac.uk
        }
\date{%
    \footnotesize 
		$^{a}$~ 
    Department of Statistical and Actuarial Sciences, University of Western Ontario, N6A 5B7 London, Canada
		\\
    $^{b}$~ Department of Mathematics, University of Bologna, Piazza di Porta san Donato, 5 Bologna
		\\
    $^{c}$~School of Mathematics, University of Edinburgh, Peter Guthrie Tait Road, Edinburgh, EH9 3FD, UK
    \\
    $^{d}$~ Centro de Matem\'atica e Aplica\c c\~{o}es (NOVA Math), FCT, UNL, 2829-516 Caparica, Portugal
    \\
    \today
}

\maketitle

\begin{abstract}
\noindent 
We propose a novel approach to numerically approximate McKean-Vlasov stochastic differential equations (MV-SDE) using stochastic gradient descent (SGD) while avoiding the use of interacting particle systems (IPS) {and the associated simulation costs required to achieve the ``propagation of chaos'' limit}. The SGD technique is deployed to solve a Euclidean minimization problem, obtained by first representing the MV-SDE as a minimization problem over the set of continuous functions of time, and then approximating the domain with a finite-dimensional subspace.
Convergence is established by proving certain intermediate stability and moment estimates of the relevant stochastic processes, including the tangent processes. Numerical experiments illustrate the competitive performance of our SGD based method compared to the IPS benchmarks. This work offers a theoretical foundation for using the SGD method in the context of numerical approximation of MV-SDEs, and provides analytical tools to study its stability and convergence.
\newline

\newline

\newline

{\noindent\textbf{Acknowledgements}: The research of SP was partially supported by the PRIN22 project no. CUP\_E53C23001670001. The author SP has also received funding from the European Union’s Horizon Europe research and innovation programme under the Marie Skłodowska-Curie Actions Staff Exchanges (Grant Agreement No. 101183168, Call: HORIZON-MSCA-2023-SE-01).
GdR acknowledges support from the \emph{Funda{\c c}\~ao para a Ci\^{e}ncia e a Tecnologia} (Portuguese Foundation for Science and Technology) through the project UIDB/00297/2020 and UIDP/00297/2020 (Centro de Matem\'atica e Aplica\c c\~oes -- NOVA Math) and from the UK Research and Innovation (UKRI) under the UK government’s Horizon Europe funding Guarantee [Project UKRI343].
}
\newline

{\noindent\textbf{Disclaimer}: Funded by the European Union. Views and opinions expressed are however those of the author(s) only and do not necessarily reflect those of the European Union or the European Education and Culture Executive Agency (EACEA). Neither the European Union nor EACEA can be held responsible for them.}
\end{abstract}




\section{Introduction}

{In this work, we propose a novel technique for numerically approximating the solution of a class of McKean-Vlasov stochastic differential equations (MV-SDEs) using stochastic gradient descent (SGD). We cast the solution of the MV-SDEs as the fixed point of a certain functional for which we find the solution via minimizing a related objective function.  As the possible domain is infinite-dimensional, we look for an approximation to the true solution in a finite-dimensional domain via the method of SGD. We show that the approximation in the finite-dimensional domain approaches the true solution in the limit.

{ MV-SDEs play an important role in a variety of fields, including, but not limited to, kinetic theory of gases and plasmas \cite{Cattiaux2008PTRFP-Probabilisticapproach,bernou2024uniform}, chemotaxis modelling \cite{suzuki2005Book}, population dynamics including crowd \cite{faure2019Book} and pedestrian \cite{MR3735068} dynamics, machine learning (ML) \cite{carmona2021convergence,carmona2022convergence}, finance and optimal control \cite{CarmonaDelarue_book_I}, quantum mechanics \cite{Golse2016meanFieldAndQuantumMechanics}, and mean-field games \cite{CarmonaDelarue_book_I}.} Typically, the solution of an MV-SDE is approximated via an $N_0$-dimensional interacting particle system (IPS) and it is well-known, via the so-called ``propagation of chaos'' {results (e.g. \cite{Sznitman91, meleard1996asymptotic})}, that the empirical law of the interacting particle system converges to the law of the MV-SDE as $N_0 \to \infty$.  Numerically, one can then use time-discretization and simulation to obtain accurate approximate solutions of the MV-SDE via the IPS (\cite{bossytalay1997,anto:koha:02,tala:vail:03}). 
The applicability of this approach, has been recently extended across a variety of settings under varying regularity and growth assumptions on its coefficients, see for example, \cite{haji2018multilevel}, \cite{sun2017ito}, \cite{chen2021SSM},\cite{chen2022euler},  \cite{biswas2022explicit}, \cite{MR4860959}, \cite{MR4908784}, \cite{yuanping2024explicit}, \cite{soni2025tamed}. 

Even though the ``IPS'' approach is very powerful and can be applied to many different classes of MV-SDEs, depending on the type of interaction kernel, the computational cost of simulating an interacting particle system can be expensive. 
As such, there have been many other works in the literature to explore alternate techniques for numerically solving for the solution of an MV-SDE. In \cite{szpruch2017iterative}, an iterative multilevel Monte Carlo scheme of creating a particle system is combined with Picard iteration to reduce the computational cost. In \cite{GobetPagliarani2018AnalysiticaApproxMVSDE}, the technique of analytical expansions is used to derive accurate approximation of the density of the solution of an MV-SDE with elliptic constant diffusion coefficient and bounded drift.  
In another approach, \cite{Baladronetal2012neuronodel}, \cite{Goddard2022}, \cite{cazacu2025stochastic} use the Fokker-Planck equation representation of the density of the solution to apply numerical techniques used to solve partial differential equations (PDEs). In \cite{AgarwalPagliarani2021FourierBased}, the authors consider a class of MV-SDEs for which the solution can be readily computed using the Picard iteration. They use a Fourier transform approach to solve a fixed point equation, which also allowed them to incorporate L\'evy jumps in their analysis in a relatively straightforward manner.  
In \cite{hutzenthaler2022multilevel}, the authors used a multilevel Picard iteration to develop numerical estimates of the solution of a class of MV-SDEs in high dimensions. They are able to prove that their numerical approach is computationally efficient and avoids the curse of dimensionality unlike the IPS approach. 
{In a way, many of the methods mentioned are of decoupling-type: in a first step the measure component (or a functional thereof) is approximated, said approximation is plugged back into the Mckean-Vlasov equation's dynamics yielding an SDE that is then approximated in the usual fashion---such approaches have been used for importance sampling \cite{dos2023importance} or to deploy Koopman transfer operator approaches to the non-linear semi-group operator of MV-SDEs \cite{ioannou2025data}.} 
In \cite{chau:garc:15} the authors developed a cubature method based approach to solve decoupled forward-backward MV-SDEs. 
{Noteworthy of mention is the recent work of \cite{chassagneux2024computing} in the context of ergodic simulation of the invariant distribution of McKean-Vlasov equations---a setting we do not explore. There, one replaces the measure argument in the coefficients of the MV-SDE by the process' occupation (empirical) measure. This transforms the MV-SDE into a \emph{self‐interacting diffusion} that has full history dependence in its coefficients. The methodology remains then  single‐trajectory (no interacting particles) and thus avoiding the classical particle‐system propagation‐of‐chaos costs. }

Recently, there have been tremendous developments in using machine learning techniques to efficiently solve nonlinear PDEs in high dimensions o
wing to the seminal work \cite{han2018solving}. Typically in this approach, an appropriate approximation of the forward process related to the nonlinear PDE and an approximation of the solution in terms of its gradient are provided as inputs to a feedforward neural network. The network is then trained to learn the mapping between the forward process and the gradient of the solution using a SGD method which minimizes a suitable loss function based on the terminal condition of the nonlinear PDE. It has been proven under different settings of nonlinear PDEs (for example, \cite{hutzenthaler2020proof}, \cite{hutzenthaler2020overcoming}) that the machine-learning-based method reliant on SGD avoids the curse of dimensionality. This approach of using machine learning to estimate the solution of a nonlinear PDE has also been successfully extended to solve McKean-Vlasov stochastic control problems, see, for example, \cite{carmona2021convergence}, \cite{carmona2022convergence}. Our proposed approach to solve an MV-SDE using SGD is inspired by these aforementioned works (and the references therein). We do not rely on a neural network to learn the solution but instead use SGD to learn optimal weights in a polynomial basis function representation of the unknown functions in the considered MV-SDE. The particular class of MV-SDE we consider in this work \eqref{eq:new_MKV} allows us to develop a more targeted approach, which is naturally more efficient than a general neural network based estimation approach.

For clarity of exposition, we work with a class of MV-SDEs in which the coefficients are separable with respect to their dependence on the measure and state of the MV-SDE{, namely
\begin{equation}
\label{eq:new_MKV_bis}
    \dd X_t = \sum_{j=1}^K \Big( \bE[\varphi_{1,j}(X_t)] \alpha_j(t,X_t)  \dd t + \bE[ \varphi_{2,j}(X_t)] \beta_{j}(t,X_t) \dd W_t \Big), \qquad 
     X_0=\xi.  
\end{equation}
Based on the recent work \cite{belomestny2018projected}, a wide class of coefficients} can be projected on a Fourier basis giving rise to the class of MV-SDEs in \eqref{eq:new_MKV_bis}. We show that the unknown functions of time $\bar\gamma_{1,j}(t):=\bE[\varphi_{1,j}(X_t)]$ and $\bar\gamma_{2,j}(t):=\bE[\varphi_{2,j}(X_t)]$ satisfy a fixed-point equation, which we then pose as a minimization problem over { the space of continuous functions}. Next, we reformulate the infinite-dimensional minimization problem as a finite-dimensional one by looking for an approximate solution in a suitable finite-dimensional subdomain, a typical choice being the space of $n$-th order polynomials. Once a basis for the latter is fixed, the objective function can be parametrized as a function on a Euclidean space. To minimize it we propose an algorithm based on SGD, where the stochastic gradient of the objective function depends on the solution to a Markovian SDE and on its gradient with respect to the parameters.   
To compute the gradient in SGD, we employ the Euler-discretization of the relevant SDEs.  To the best of our knowledge, our approach is the first to deploy an SGD-type algorithm to solve MV-SDEs. Moreover, we also provide a ``mini-batch'' version of our SGD based algorithm that greatly improves the efficiency of the algorithm. 

In this work, we essentially build a theoretical foundation for the study of SGD approach applied to numerical approximation of MV-SDEs. We provide a convergence study and establish the related properties for such purpose. Concretely, the convergence of our method is reliant on stability estimates of the related Markovian SDE and tangent processes, including a general interest result on the higher-order time-regularity property of the maps $t \mapsto \bE[\varphi_{\cdot,\cdot}(X_t)]$. 
The moment estimates and Lipschitz property of the tangent process with respect to the unknown function also allows us to establish an explicit formula for the gradient of the loss function. Moreover, by performing a suitable penalization of the loss function, we are able to obtain Lipschitz continuity of its gradient. These intermediate results allow us to use the existing convergence results of descent methods by \cite{bertsekas2000gradient} and \cite{patel2021stochastic} to show convergence of our algorithm to a stationary point. { \label{aux:RandomPageLabel-page3}As our minimisation problem is non-convex, classical results that ensure convergence of SGD algorithms to the global minimum do not apply to our setting. On the contrary, we rely on the developments in \cite{bertsekas2000gradient} and \cite{patel2021stochastic} for the convergence of SGD algorithms to stationary points, in non-convex settings. We observe that the study of gradient descent algorithms is an active research field, also with attempts of proving convergence to the global minimum in the non-convex setting. For instance, in \cite{fehrman2020convergence}, the authors prove global convergence for a particular version of the SGD algorithm that makes use of mini-batches and randomized initial point, under suitable geometric assumptions for the objective function. Checking these assumptions in our setting would be quite complex, thus, for the purpose of readability, we only verify the assumptions in \cite{bertsekas2000gradient} and \cite{patel2021stochastic} to ensure convergence to a stationary point.} Through extensive numerical studies, we demonstrate the excellent performance of our SGD algorithm in solving well-studied MV-SDEs. In particular, we measure the accuracy of our method by computing relative error with respect to the benchmarks obtained by the IPS approach, and show that we achieve very low error values with a small computational budget.  

Overall, the results of this paper can be viewed as a proof of concept for the application of stochastic gradient descent algorithms as a valid tool to deliver numerical approximations of MV-SDEs without the simulation of IPSs. 
{ 
In order to reduce the technicalities and keep the proofs as transparent as possible, the convergence analysis is performed here in the case of bounded coefficients. However, it could be extended to the case of $\alpha_j(t,\cdot)$, $\beta_j(t,\cdot)$ with sub-linear growth and $\varphi_{1,j}$ and $\varphi_{2,j}$ with polynomial growth. We discuss this issue further below {(see Remark \ref{rem:sublin})}. In particular, the numerical tests presented in Section \ref{sec:quadratic} confirm that the approach is viable also when the coefficients $\varphi_{1,j}$ and $\varphi_{2,j}$ have quadratic growth.} 
An extension of this methodology to the case of H\"older continuous drift coefficients seems feasible, up to dealing with some technical issues arising from the fact that the gradient process becomes a diffusion with singular (distributional) drift coefficients. 
We also expect this approach to be suitable for extensions to the jump-diffusion framework. 
Finally, an interesting and challenging extension is in regards the adapting this methodology to frameworks where the simulation of the associated IPS is less amenable, e.g. conditional MV-SDEs or MV-SDEs with moderate interaction among others.

The rest of the paper is organised as follows. At the start of Section \ref{sec:Seperablecoefficients}, we provide the main setting of our considered MV-SDEs and { for convenience of the reader Table \ref{tab:notation} and Table \ref{tab:notation222} summarize the notation used throughout. }  
We  formulate the fixed-point problem for the unknown measure-dependent function in the MV-SDEs and then cast it as a minimization problem over the space of continuous functions. In Section \ref{sec:ProjOnFiniteDimSpace}, we introduce the finite-dimensional approximation of the original domain and provide an upper bound for the error we incur by solving the minimization problem on the reduced domain.  
We propose our SGD based algorithm to solve the minimization problem at the beginning of Section \ref{sec:sgd for mvsde}, where we also provide a ``mini-batch'' version of it. The convergence proof of the algorithm is presented in Section \ref{sec:convergence} along with the necessary intermediate results. The results from the extensive set of numerical experiments are presented in Section \ref{sec:NumericalStudy}. 
The proofs of all the main theoretical results are relegated to the Appendix.

\begin{table}[H]
{\begin{center}
\begin{tabular}{|c|l|}
        \hline
        \textbf{Notation} & \textbf{Description} \\ 
        \hline
        $\mathbb{N}$ & Set of natural numbers starting at 1 \\ 
        $\mathbb{N}_0 := \mathbb{N} \cup \{0\}$ & Set of natural numbers including 0 \\ 
        $\mathbb{R}$ & Set of real numbers \\ 
        $\mathbb{R}_+ := [0,\infty)$ & Set of non-negative real numbers \\ 
        $\langle x, y \rangle$ & Euclidean scalar product for $x, y \in \mathbb{R}^m$ \\ 
        $|x|$ & Euclidean norm for $x \in \mathbb{R}^m$ \\ 
        $\text{Tr}(A)$ & Trace operator of a square matrix $A$ with real entries
				\\ 
        $L^p([0,T],\mathbb{R}^m)$ & (or $L^p([0,T])$ in short) is the Space of Borel-measurable functions  \\
        & $f: [0,T] \to \mathbb{R}^m$ with $\| f \|_{L^p([0,T])} = \left(\int_0^T |f(t)|^p dt \right)^{\frac{1}{p}} < \infty,$ for $p \geq 1$. 
				\\ 
        $L^\infty([0,T],\mathbb{R}^m)$ &  (or $L^\infty([0,T])$ in short) is the Space of Borel-measurable functions \\ 
         & $f: [0,T] \to \mathbb{R}^m$ with   $\| f \|_{L^\infty([0,T])} = \sup_{t \in [0,T]} |f(t)| < \infty$
				\\ 
        $C([0,T],\mathbb{R}^m)$ & Space of continuous functions $f: [0,T] \to \mathbb{R}^m$ endowed with supremum \\
        &norm  $\| f \|_{\infty} = \sup_{t \in [0,T]} |f(t)|$
				\\
        $C^n_b([0,T])$ 
				& Space of $n$-times continuously differentiable functions $f:[0,T] \to \Rb^m$\\
        & with bounded derivatives
				\\
        $C_b^{n,q}([0,T] \times \mathbb{R}^d)$ & Space of continuous functions $f= f(t,x):[0,T] \times \mathbb{R}^d \to \Rb^m$, $n$-times\\
        & differentiable in $t$, $q$-times in $x$, with bounded and continuous derivatives\\
        $\partial_{x_i} f$ & Partial derivative of $f= f(x) : \Rb^m \to \Rb^n$ with respect to component $x_i$\\
        $\nabla_x f$ & Jacobian $(n \times m)$-matrix. If $n=1$, it is the gradient vector\\
        $H_x f$ & Hessian $(n \times m \times m)$-tensor of $f$. 
        If $n=1$, it is the Hessian matrix\\
        $(\Omega, \cF, (\cF_t)_{t\in[0,T]},\bP)$ & Filtered probability space, with $(\cF_t)_{t\in[0,T]}$ satisfying the usual\\
        &assumptions of completeness and continuity\\ 	
        $\mathbb{E}[\cdot]$ & Expectation operator under measure $\bP$ \\
        $L^p$ & Space of random variables $\xi$ such that $\|\xi\|_{L^p}=\mathbb{E}[|\xi|^p]^{\frac1p} < \infty,$  for $p \geq 1$
				\\
        $\mathbb{L}^p_{[0,T]}$ 
				& Space of progressively measurable processes $X$ such that \\
				& $\|X\|_{\mathbb{L}^p_{[0,T]}}=\bigl( \int_0^T \mathbb{E}[|X_t|^p] dt\bigr)^{\frac1p} < \infty$, for $p \geq 1$
				\\
	$\cS^{p}_{[0,T]}$ %
	& Space of progressively measurable processes $X$ such that\\
	& $\|X\|_{\cS^{p}_{[0,T]}}=\mathbb{E} [ \sup_{t \in [0,T]} |X_t|^p ]^{\frac1p} < \infty,$  for $p \geq 1$\\
        $\cS^{\infty}_{[0,T]}$ 
	& Space of progressively measurable processes $X$ such that\\
        & $\|X\|_{\cS^{\infty}_{[0,T]}}=\esssup_{\omega \in \Omega} \sup_{0 \leq t \leq T} |X_t| < \infty$\\
        \hline
    \end{tabular}
    \end{center}\caption{Main notation for spaces and classic operators.}\label{tab:notation}}
    \end{table}

\small
\begin{table}[H]
{\begin{center}
\begin{tabular}{|c|l|}
        \hline
        \textbf{Notation} & \textbf{Description} \\ 
        \hline
        $X$ & Solution of the reference MV-SDE in \eqref{eq:new_MKV}\\
$Z^{\gamma}$ &  Solution of \eqref{eq:SDE_Z} given an input function $\gamma \in L^{\infty}([0,T],\bR^K)$ 
\\
 $\Lc  $ & Lifting map \eqref{eq:lifting_operator}, $ \mathbb{R}^{(n+1)\times K} \ni a \mapsto \Lc a :=  a_0g_0+   \cdots + a_n g_n \in L^\infty([0,T])$, \\ 
         & with $g_0,\cdots,g_n$ being linearly independent vectors of $C([0,T],\bR)$
				\\
$F$ & Functional \eqref{eq:minimization} to be minimized \\ 
$\h$, $H$ & Auxiliary truncation and penalization maps \eqref{eq:h} and \eqref{eq:Hn}\\ 
$G$ & Functional \eqref{eq:FunctionalG} (see also \eqref{eq:min_Sn}) to be minimized (built from $F$)\\ 
$Z^{a}$ &  Solution of  \eqref{eq:SDE_Z_ter}, i.e. $Z^\gamma$ in \eqref{eq:SDE_Z}  when $\gamma =\h\circ \Lc a$ for given $a \in \mathbb{R}^{(n+1)\times K}$
\\
$Y^{a;k,j}$, $Y^{k,j}$ & Tangent process $\partial_{a_{k,j}} Z^a$ in \eqref{eq:derivative_Z} 
\\
$v$, $v_{k,j}$ & Gradient functional for the SGD algorithm in \eqref{eq:v_function}\\
        \hline
    \end{tabular}
    \end{center}\caption{Main notation for primary processes and auxiliary functions.}\label{tab:notation222}}
    \end{table}
\normalsize

\section{McKean-Vlasov SDEs with separable coefficients}
\label{sec:Seperablecoefficients}
Let $T>0$ be a fixed time horizon. All the stochastic differential equations in this article are considered up to time $T$.
We consider the class of $\Rb^d$-valued MV-SDEs in \eqref{eq:new_MKV_bis}, where the interaction through the state is separated from the interaction through the law. 
Their dynamics can be compactly written as
\begin{equation}
\label{eq:new_MKV}
\begin{cases}
    \dd X_t = \bar\gamma(t) \big( \alpha(t,X_t)  \dd t +  \beta(t,X_t) \dd W_t \big), \qquad 
     X_0=\xi,\\
 \bar\gamma(t)= \bE[\varphi(X_t)], 
    \end{cases}    
\end{equation}
for measurable maps $\alpha: [0,T] \times \Rb^d \to \mathbb{R}^{K\times d}$, $\beta: [0,T] \times \Rb^d  \to \mathbb{R}^{K \times d \times q}$, $\varphi: \Rb^d \to \Rb^K $, $\xi$ a random variable, and with $W$ being a $q$-dimensional Brownian motion (column vector). The full probabilistic framework can be found in the Appendix.

\begin{remark}
\label{rem:remark123}
The framework \eqref{eq:new_MKV} of $\bar\gamma$ maps can be extended to the type  $\bar\gamma(\cdot)= h \big( \bE[\varphi(X_\cdot)] \big)$ for some invertible and suitably regular function $h$. We present our results under the setting of \eqref{eq:new_MKV} only for the sake of simplicity. 
\end{remark}

Our starting point could also have been a MV-SDE with more general coefficients, to which we could have applied a projection over a Fourier basis (see \cite{belomestny2018projected}), but such an exercise would still have resulted in the class of MV-SDEs considered here. { For example, in Section \ref{sec:convol} we show, following \cite[Section 2]{belomestny2018projected}, that an MV-SDE with drift coefficient given by $\mathbb{E}\left[\exp\left(-{(X_t- x)^2}/{2}\right)\right]|_{x=X_t}$ can be approximated by a system in the form of \eqref{eq:new_MKV}, to which we apply our SGD algorithm.}

To proceed further, we first present the well-posedness and regularity assumptions.

\begin{assumption}
\label{ass:initial}
The random variable $\xi\in L^2$.  
\end{assumption} 
\begin{assumption}
\label{ass:regularity-for-SDE-no-basis}
The functions $\alpha,\beta\in C([0,T]\times \Rb^d)$ 
and there exists $\RR>0$ such that 
\begin{equation}\label{eq:lip_assump}
| \alpha(t,x) - \alpha(t,y) | + | \beta(t,x) - \beta(t,y) | + |\varphi(x) - \varphi(y) |  \leq \RR  |x-y| , \qquad   x,y\in\Rb^d ,  \ t\in [0,T],  
\end{equation}
and $\alpha$, $\beta$ and $\varphi$ are bounded by $\RR$. 
\end{assumption}


The following well-posedness result is classical.
\begin{proposition}
\label{prop:well-posedness-of-X}
Let Assumption \ref{ass:initial} and \ref{ass:regularity-for-SDE-no-basis} hold true. Then, 
there exists a unique solution $X = (X_t)_{t\in [0,T]}$ to MV-SDE \eqref{eq:new_MKV} in $\cS^2_{[0,T]}.$ Furthermore, $\bar\gamma\in C_b([0,T])$.
\end{proposition}
The above result follows from \cite[Theorem 4.21 (p.236)]{CarmonaDelarue_book_I}). { In particular, t}he fact that $\bar\gamma$ is bounded and continuous is straightforward as $\varphi$ is bounded and continuous.

{\begin{remark}[Growth assumptions]\label{rem:sublin}
We expect the conclusion of Proposition \ref{prop:well-posedness-of-X} to still hold true when replacing the condition in Assumption \ref{ass:regularity-for-SDE-no-basis} of boundedness on $\alpha(t,\cdot)$ and $\beta(t,\cdot)$ with sub-linear growth, the latter implied by the Lipschitz condition \eqref{eq:lip_assump}. However, we are not aware of this result in the literature. The proof would take advantage of the specific structure of the coefficients, which are separable with respect to the state and measure component. Also, we expect the result to be true by assuming the function $\varphi$ only locally Lipschitz continuous with polynomial growth, provided the initial datum has suitably many finite moments.
\end{remark}}

{\begin{remark}[Regularity assumptions]\label{rem:regularityAssumptions}
Concerning the well-posedness of MV-SDEs and the moment estimates of its solutions, Lipschitz continuity of the coefficients is a rather standard assumption. The literature on MV-SDEs also contains results for less regular coefficients (see \cite{de2020strong} and the references therein), but these typically require additional structural assumptions, e.g. ellipticity conditions for the diffusion coefficient. Also, we rely on the Lipschitz continuity of $\alpha(t,\cdot)$ and $\beta(t,\cdot)$ in the convergence analysis of the SGD algorithm (see Section \ref{alg:SGDMVSDE} below). 
Thus, we prefer to introduce now this regularity requirement.
\end{remark}}

{\begin{example}\label{example:onCoefficients}
Assumption \ref{ass:regularity-for-SDE-no-basis} requires the coefficient functions $\alpha(t,\cdot)$, $\beta(t,\cdot)$ and $\varphi$ to be bounded and Lipschitz continuous. Among the elementary functions, this class includes several trigonometric functions (e.g. the model in Section \ref{sec:kuramoto}), bounded rational functions, Gaussian-type functions, as well sums and products of these. On the other hand, polynomial functions do not satisfy Assumption \ref{ass:regularity-for-SDE-no-basis}. However, affine functions $\alpha(t,\cdot), \beta(t,\cdot)$ and  polynomial type functions $\varphi$ are allowed in the extended framework discussed in Remark \ref{rem:sublin} (e.g. the model in Section \ref{sec:quadratic}).
\end{example}}


We now establish a time-regularity result for $\bar\gamma$, which can be achieved by requiring more regularity on the function $\varphi$ and on the coefficients $\alpha,\beta$. { We will see in Section \ref{sec:ProjOnFiniteDimSpace} that one can take advantage of this higher regularity to have a faster convergence of the solution to the finite-dimensional minimization problem to the true $\bar\gamma$ (see Proposition \ref{prop:opt real seq} below).}
\begin{assumption}\label{assum:extra_reg}
There exists $N\in \mathbb{N}$ and $\RR>0$ such that $\alpha,\beta \in C_b^{N-1,2(N-1)}([0,T] \times \mathbb{R}^d)$ and  
$\varphi\in C^{2 N}_b(\mathbb{R}^d)$, with $\alpha$, $\beta$, $\varphi$ and all their derivatives bounded by $\RR$.
\end{assumption}

\begin{proposition}[Higher-order smoothness of $\bar\gamma$]
\label{prop:regularity_gamma}
Let Assumption \ref{ass:initial}, \ref{ass:regularity-for-SDE-no-basis} and \ref{assum:extra_reg} with some $N\in\mathbb{N}$ hold true. 
Then, the function $\bar\gamma$ defined in \eqref{eq:new_MKV} is in $C^N_b([0,T])$ and, for any $m=1,\cdots, N$, we have  
\begin{equation}\label{eq:represent_deriv_gamma}
\frac{\dd^m}{\dd t^m} \bar\gamma(t) = \Eb\big[ \big(\Kc^m \varphi\big) ( t, X_t ) \big], \qquad t\in [0,T], 
\end{equation}
where $\Kc$ is the Kolmogorov operator acting on a function $u\in C_b^{1,2}([0,T] \times \mathbb{R}^d)$ as
\begin{equation}
\Kc u (t,x) =  \partial_t u + \frac{1}{2} \textrm{Tr} \Big(  
\bar\gamma(t)   \beta(t,x)\beta^\top(t,x)\bar\gamma^\top(t) H_x u (t,x)
  \Big) + \big\langle \bar\gamma(t) \alpha(t,x) ,  \nabla_x u (t,x) \big \rangle ,
\end{equation}
with $\Kc^m u$ being the $m$-times composition of $\Kc$ acting on $u$ as 
\begin{equation}
[0,T] \times \mathbb{R}^d \ni (t,x)\mapsto 
\big(\Kc^m u\big)(t,x) =\underbrace{ \Kc \big( \cdots \Kc(\Kc u) \cdots \big)}_{m \text{ times}}(t,x).
\end{equation} 
\end{proposition}
The proof is found in Appendix \ref{sec:proofs}. 

{\begin{remark}\label{rem:ext_reg_der}
By inspecting the proof of Proposition \ref{prop:regularity_gamma} and by Remark \ref{rem:sublin}, one can see that the same conclusion could be derived lifting the boundedness condition in Assumption \ref{assum:extra_reg}, if the derivatives of $\alpha(t,\cdot)$, $\beta(t,\cdot)$ and $\varphi$ have polynomial growth, provided that $X_0$ (thus, $X_t$) has suitably many finite moments.
\end{remark}
\begin{example}\label{example:coefficients}
All the elementary functions satisfying Assumption \ref{ass:regularity-for-SDE-no-basis} discussed in Example \ref{example:onCoefficients} are smooth with bounded derivatives and thus fulfil Assumption \ref{assum:extra_reg}. On the other hand, affine $\alpha(t,\cdot), \beta(t,\cdot)$ and polynomial $\varphi$ are allowed in the extended framework discussed in Remark \ref{rem:ext_reg_der}.
\end{example}
}
 

\subsection[Representation of the gamma map]{Fixed-point equation for $\bar\gamma$}
\label{sec:ApproxViaSGD}

We introduce a closely related SDE to \eqref{eq:new_MKV} which takes a function $\gamma$ as input. Namely, given $\gamma\in L^{\infty}([0,T],\bR^K) $ denote by $Z^{\gamma}$ the solution to the Markovian SDE,
\begin{equation}\label{eq:SDE_Z}
    \dd Z_t = \gamma(t)\big( \alpha(t,Z_t)  \dd t +  \beta(t,Z_t) \dd W_t \big),\qquad Z_0=\xi.
\end{equation} 
The following result is a direct application of Lemma \ref{lem:sde_stability_general}.
\begin{proposition}
\label{prop:parameterestimates}
Let Assumption \ref{ass:initial} and \ref{ass:regularity-for-SDE-no-basis} hold true for some $\RR>0$. Then \eqref{eq:SDE_Z} is strongly well-posed for any $\gamma\in L^{\infty}([0,T],\bR^K) $. Furthermore, for any $\kappa>0$, there exists $C>0$ dependent on $\kappa, T,\RR$  
and $K,$ 
such that
\begin{align}
\label{eq:ErrorIneq-X-gamma-prime}
        \lVert Z^{\gamma} - Z^{\gammat} \rVert_{\cS^2_{[0,T]}} 
        \leq C  
         \|\gamma - \gammat\|_{\infty},
    \end{align}
for any  $\gamma,\gammat \in L^{\infty}([0,T],\bR^K)$ such that $\|\gamma\|_{\infty},\|\gammat\|_{\infty} \leq 
\kappa$.
\end{proposition}

Now, define the operator $\Psi:C\big([0,T],\bR^K) \to C\big([0,T],\bR^K)$ given as 
\begin{equation}\label{eq:SDE_Psi of gamma}
    \Psi(\gamma)(t) : =  \Eb [ \varphi(Z^{\gamma}_t)] , \qquad t\in[0,T], 
\end{equation}
which is well-defined under Assumptions \ref{ass:initial} and \ref{ass:regularity-for-SDE-no-basis}. 

Let us start by observing that solving the MV-SDE \eqref{eq:new_MKV} is equivalent to  finding a fixed point for $
\Psi$. Indeed, if $\bar\gamma$ is a solution to 
\begin{equation}\label{eq:fixed_point}
\gamma = 
 \Psi(
 \gamma) ,
\end{equation}
then the process $Z^{\bar\gamma}$ solves \eqref{eq:new_MKV} and vice-versa. 
Also note that,  
if $\bar\gamma$ solves \eqref{eq:fixed_point}, then $\bar{\gamma}$ solves the minimization problem
\begin{equation}
\label{eq:minimization}
\min_{\gamma\in C([0,T])} F^2(\gamma),
\end{equation}
where 
\begin{equation}\label{eq:def_F}
F : C\big([0,T],\bR^K\big) \to \bR,\qquad F(\gamma): =  \|  \Psi(\gamma) - \gamma  \|_{L^2_{([0,T])}}.
\end{equation}
Therefore, under Assumption \ref{ass:initial} and \ref{ass:regularity-for-SDE-no-basis}, Proposition \ref{prop:well-posedness-of-X} implies that there exists a unique $\bar\gamma \in C\big([0,T],\bR^K\big)$ that solves \eqref{eq:minimization}, in particular, 
\begin{equation}
\label{eq:optimal gamma}
\min_{\gamma\in C([0,T])} F^2(\gamma) = F^2(\bar\gamma)  = 0.
\end{equation}

\begin{lemma}\label{lem:F_lipschitz}
Let Assumption \ref{ass:initial} and \ref{ass:regularity-for-SDE-no-basis} hold true for some $\RR>0$. Then, the functional $F$ in \eqref{eq:minimization} is locally Lipschitz continuous in sup-norm $\|\cdot\|_\infty$. Precisely, for any $\kappa>0$ there exists $C>0$ such that 
\begin{align}
\label{eq:Lip_f}
| F(\gamma) - F(\hat\gamma)| \leq C \lVert \gamma-\hat\gamma \rVert_{\infty}, \qquad \gamma, \hat\gamma \in   C([0,T],\bR^K), \quad \text{with }  \|\gamma\|_{\infty},\|\hat\gamma\|_{\infty} \leq \kappa,
\end{align}
where $C>0$ depends on $\kappa, T,\RR$ 
and $K$.
\end{lemma}
\begin{proof}
Given $\kappa>0$, take $\gamma, \hat\gamma \in   C([0,T],\bR^K) \ \text{with }  \|\gamma\|_{\infty},\|\hat\gamma\|_{\infty} \leq \kappa$. By the reverse triangle inequality we have     \begin{equation}
\label{eq:lipschitzness-of-f}
\Big|\lVert \Psi(\gamma) - \gamma \rVert_{L^2([0,T])} - \lVert \Psi(\hat\gamma) - \hat\gamma \rVert_{L^2([0,T])} \Big| \leq \lVert \Psi(\gamma) - \Psi(\hat\gamma) \rVert_{L^2([0,T])} + \lVert \gamma - \hat\gamma \rVert_{L^2([0,T])}.
\end{equation}
The second term in the right hand side of \eqref{eq:lipschitzness-of-f} is dominated by $\sqrt{T}    \lVert \gamma- \hat\gamma\rVert_{\infty}.$ For the first term, by the definition of $\Psi$ and by the Lipschitz continuity of $\varphi$, we obtain
    \begin{equation}
        \lVert \Psi(\gamma) - \Psi(\hat\gamma) \rVert^2_{L^2([0,T])} 
        =
      \int_0^T\Big| \bE\big[\varphi(Z^\gamma_t) - \varphi(Z^{\hat\gamma}_t)\big]\Big|^2 \dd t 
 \leq
        C \int_0^T\bE\Big[\big|Z^\gamma_t - Z^{\hat\gamma}_t\big|^2\Big] \dd t  
        \leq
        C\, T \lVert \gamma-\hat\gamma \rVert^2_{\infty},
    \end{equation}
    where the last inequality follows from Proposition \ref{prop:parameterestimates}. Thus, by combining the two bounds for the terms in the right hand side of \eqref{eq:lipschitzness-of-f}, we get the result in \eqref{eq:Lip_f}.
\end{proof}
A quick inspection of the above proof shows that \eqref{eq:Lip_f} also holds with $\|\cdot\|_\infty$ replaced by $\|\cdot\|_{L^2([0,T])}$ since the same is also true for Proposition \ref{prop:parameterestimates}.

\begin{remark}
\label{Rem:More weight towards T}
Depending on the numerical accuracy requirements, instead of $F(\gamma)$ defined above, we could also use an objective function $\widehat F(\gamma)$ given as 
\begin{align*}
\widehat F^2(\gamma) = \int_0^T w(s)|\Psi(\gamma)(s) - \gamma(s)|^2 ds, 
\end{align*}
with a kernel (weight) function $w$ over $[0,T]$. 
For example, we can take $w(t)=C_1e^{C_2 t}$ for $C_1,C_2>0$.  
Such a choice of objective function $\widehat F(\gamma)$ in the minimization problem \eqref{eq:minimization}
would force the stochastic gradient descent algorithm to be more accurate towards the end of the interval $[0,T]$. 
\end{remark}

\subsection{Projection on a finite dimensional domain}
\label{sec:ProjOnFiniteDimSpace}

In order to solve \eqref{eq:minimization} {numerically} using a gradient descent method, {  it is imperative that }we first narrow down the infinite dimensional domain $C([0,T],\bR^{K})$ to a finite dimensional one. {This approach is a natural choice in our setting of solving \eqref{eq:minimization} as the solution $\bar \gamma$ is a continuous map.} Hereafter, we fix $n\in\mathbb{N}$ and a family $g_0,\cdots,g_n$ of linearly independent vectors of $C([0,T],\bR)$.   
Define the finite-dimensional sub-space of $C([0,T],\bR^{K})$ as
\begin{equation}\label{eq:Sn}
    S_n := \{ a_0 g_0 + \cdots + a_n g_n: a=(a_0, \cdots, a_n) \in \mathbb{R}^{(n+1) K   } \},
\end{equation}
and the {distance between $\gamma  \in C([0,T],\bR^{K})$ and $S_n$} as
\begin{equation}\label{eq:defdn}
d_n(\gamma) := \inf\{ \| \gamma - p \|_{\infty} : p \in S_n  \}.
\end{equation} 
Also introduce 
a linear lifting operator 
\begin{equation}\label{eq:lifting_operator}
 \Lc : \mathbb{R}^{(n+1)\times K} \to S_n, \qquad    \Lc a = \Lc (a_0, \cdots, a_n) := a_0g_0+   \cdots + a_n g_n.
\end{equation}

\begin{remark}\label{rem:distance}
The operator $\Lc$ is an isomorphism between finite-dimensional normed vector spaces. In particular, both $\Lc$ and its inverse $\Lc^{-1}$ are continuous and bounded operators. Furthermore, by a compactness argument, it can be easily seen that for any $\gamma \in C([0,T],\bR^K),$ there exists $p_n(\gamma) \in S_n$ such that
\begin{equation}\label{eq:min_dist_polyn}
    \| \gamma - p_n(\gamma) \|_{\infty} = d_n(\gamma). 
\end{equation}
In other words, $d_n$ is well-defined as a minimum and such minimum is attained at $p=p_n(\gamma)$. { It is important to observe that both $p_n(\gamma)$ and $d_n(\gamma)$ only depend on $S_n$ and not on the particular choice of the basis $g_i$, $i=0,\dots,n$.}
\end{remark}

\begin{example}[Polynomial of best approximation]
\label{ex:polynomials}
Consider $n\in\mathbb{N}$ and $g_i (t) = t^i$ for $i=0,\cdots, n$. Then the elements of $S_n$ read as
\begin{equation}
    p(t) = \sum_{i=0}^n a_i t^i, \qquad t\in[0,T].
\end{equation}
{ This choice clearly leads to $S_n = \mathcal{P}_n(\mathcal{R}^K)$, the space of $\mathbb{R}^K$-valued polynomials of order $n$. We stress that $S_n = \mathcal{P}_n(\mathcal{R}^K)$ could be also obtained by fixing a basis $g_i$, $i=0,\dots,n$, which is not the canonical one. Possible choices include Legendre, Chebyshev, Laguerre and Hermite polynomials \cite{MR2006500Suli}. In particular, the numerical results presented in Section \ref{sec:NumericalStudy} refer to the specific choice of $g_i$ given by the Lagrange polynomials centered at Chebyshev points, i.e.
\begin{equation}\label{eq:orthog poly basis}
g_i(t) = \Bigg(   \prod_{\substack{ 0\leqslant l \leqslant n \\ l\neq i}} \frac{t-t_l}{t_i - t_l}  \Bigg), \qquad t_i = \frac{T}{2} + \frac{T}{2} \cos\Bigl( \frac{2i+1}{2n+2}\pi\Bigr),\qquad i=0,1,\ldots,n.
\end{equation}

Finally note that, in particular, the distance $d_n(\gamma)$ and the projection $p_n(\gamma)$ (see Remark \ref{rem:distance}) are the same for all these choices of basis, for they all generate the same finite dimensional space $S_n = \mathcal{P}_n(\mathcal{R}^K)$. In this case, $p_n$ in \eqref{eq:min_dist_polyn} is called the \emph{polynomial of best approximation of $\gamma$} (see \cite[Chapter 5]{isaacson2012analysis}).} 
\end{example}
\begin{example}
\label{ex:spline}
Other examples for the basis $\{g_i, i = 0, 1, \ldots\}$ include 
splines \cite{isogeometricMethods} and smooth Fourier wavelet expansions \cite{MR2466138Oosterlee2008}.   
\end{example} 

Since the maps $\gamma$ are bounded by $\|\varphi\|_\infty$, we introduce the auxiliary map $\h:\Rb^{K}\to\Rb^K$ so that we can ensure that $\h(\Lc a)$ is also bounded. Concretely, let $\h$ be a smooth bounded function $\Rb^{K}\ni x \mapsto \h(x)=\big(\h_1(x_1),\cdots,\h_K(x_K)\big)$ defined as
\begin{equation}
\label{eq:h}
\Rb^K \ni x \mapsto \h(x)= 
\begin{cases}
x ,& \text{if}\quad |x| \leq \| \varphi \|_{\infty} + 2 d_n (\bar\gamma)
\\
\| \varphi \|_{\infty} + 2 d_n (\bar\gamma),& \text{if}\quad |x| \geq \| \varphi \|_{\infty} + 2 d_n (\bar\gamma)
\end{cases},
\end{equation}
where it is clear that 
{$\|\h\|_{\infty} \leq \| \varphi \|_{\infty} + 2 d_n (\bar\gamma)$}.

Let the constant $\rad_n>0$ be such that 
\begin{equation}
B_{\rad_n}\supset \Lc^{-1}\big( B_{\| \varphi \|_{\infty} + d_n (\bar\gamma)} \big),
\end{equation}
where $B_{\rad_n}$ and $B_{\| \varphi \|_{\infty} + d_n (\bar\gamma)}$ denote the balls in $\mathbb{R}^{(n+1) K}$ and  $S_n$, respectively, with radius $\rad_n$ and $\| \varphi \|_{\infty} + d_n (\bar\gamma)$,  centered at zero. Note that such $\rad_n$ exists as $\Lc^{-1}$ is a linear, and thus a bounded operator (see Remark \ref{rem:distance}).
Finally, we consider a non-negative real-valued function $H
\in C^1(\mathbb{R}^{(n+1) K})$ such that 
\begin{equation}\label{eq:Hn}
\text{supp }H
 \subset B^c_{\rad_n}
\quad\textrm{and}\quad 
 \lim_{|a|\to + \infty} H
 (a)=+\infty.
\end{equation}
\begin{example}\label{ex:H}
A function $H
$ satisfying hypothesis \eqref{eq:Hn} above is given by
\begin{equation}
H
(a):= 
\begin{cases}
0,								& \text{if } a \in B_{\rad_n}\\
(|a|-\rad_n)^2,	& \text{otherwise}
\end{cases} , \qquad { \nabla H
(a):= 
\begin{cases}
0,								& \text{if } a \in B_{\rad_n}\\
\frac{2 (|a| - \rad_n)}{|a|}a,	& \text{otherwise}
\end{cases} ,}
\end{equation}
\end{example}
\begin{remark}\label{rem:hH}
Let $\bar\gamma \in C([0,T],\bR^K)$ be the unique minimizer to \eqref{eq:minimization} and consider the element $p_n(\bar\gamma)$ introduced in Remark \ref{rem:distance}. Since $\|\bar\gamma  \|_{\infty} \leq \| \varphi \|_{\infty}$, it is clear that $\|p_n(\bar\gamma)\|_{\infty}\leq \| \varphi \|_{\infty} + d_n(\bar\gamma)$. Therefore, by \eqref{eq:h} and \eqref{eq:Hn}, we obtain that
\begin{equation}
\h
\big( p_n(\bar\gamma) \big) = p_n(\bar\gamma)
\quad\textrm{and}\quad 
H
\big(\Lc^{-1}  p_n(\bar\gamma)\big)\equiv 0.
\end{equation}
\end{remark}
In place of \eqref{eq:minimization}, we consider the following optimization problem
\begin{equation}
\label{eq:min_Sn}
  \min_{a\in \mathbb{R}^{(n+1) K}} G(a)
	\quad\textrm{with}\quad 
	G(a):=	F^2(\h\circ \Lc a) + H
  (a).
\end{equation}

The role of map $\mathbf{h}$ is to ensure that the estimate of $\bar{\gamma}$, defined in \eqref{eq:optimal gamma}, remains bounded when computing the loss function $F$. This is reflected when we compute $F^2(\h\circ \Lc a)$ in \eqref{eq:min_Sn}. Moreover, the function $H$ penalizes the coefficient vector $a$ when the norm of our approximation $\mathcal{L}a$ grows beyond $\| \varphi \|_{\infty} + d_n (\bar\gamma).$ This is reflected in the definition of $H$ in \eqref{eq:Hn} and also motivated in Remark \ref{rem:hH}. 

\begin{lemma}\label{lemm:err_reduction}
Let Assumptions \ref{ass:initial} and \ref{ass:regularity-for-SDE-no-basis} 
hold true for some $\RR>0$. Let $\bar\gamma \in C([0,T],\bR^K)$ be the unique solution to \eqref{eq:minimization} and assume that $\| \varphi \|_{\infty} + d_n(\bar\gamma)\leq \kappa$ for some $\kappa>0$. 
Then the minimization problem \eqref{eq:min_Sn} admits a solution. Furthermore, for any 
\begin{equation}
\bar a_n \in \argmin_{a\in \mathbb{R}^{(n+1) K}} G(a) 
\end{equation}
we have
\begin{equation}\label{eq:estimate_domain_red}
F(\Lc \bar a_n)   \leq C\,  d_n(\bar \gamma) , 
\end{equation}
where $C>0$ depends on $\kappa, T,\RR$ 
and the dimensions.
\end{lemma}
\begin{proof}
The boundedness of $\varphi$ implies that $F^2(\h\circ \Lc a)$ is bounded. Thus, by \eqref{eq:Hn}, 
$G(a)$ tends to $+\infty$ as $|a| \to \infty$. 
On the other hand, Lemma \ref{lem:F_lipschitz} together with the continuity of $\h$ and $H$  imply that 
$G$ is continuous on $\mathbb{R}^{(n+1) K}$. Therefore, the fact that 
$G$ has at least a minimum $\bar a_n$ on $\mathbb{R}^{(n+1) K}$ stems from Weierstrass' extreme value theorem. 
 
Furthermore, let $\bar a_n$ be a minimiser of $G$ 
we then have 
    \begin{align}
 F^2(\h\circ \Lc \bar a_n) 
\leq  F^2(\h\circ \Lc\bar  a_n)    + H_n(\Lc\bar a_n)
& \leq F^2\big(\h\circ p_n(\bar\gamma)\big)    + H
 \big(\Lc^{-1}  p_n(\bar\gamma)\big) 
\\ 
  & = F^2\big(p_n(\bar\gamma)\big) = \big|F\big( p_n(\bar\gamma)\big) - F(\bar \gamma)\big|^2 
\\
          & \leq C \, \lVert  p_n(\bar\gamma) - \bar \gamma \rVert^2_{\infty} = C \, d^2_n(\bar \gamma) .
  \end{align}
In the above, we used Remark \ref{rem:hH} and Lemma \ref{lem:F_lipschitz}. The result follows since $\| \bar\gamma \|_{\infty}\leq \| \varphi \|_{\infty}$ and $\| p_n(\bar\gamma) \|_{\infty}\leq \| \varphi \|_{\infty} + d_n(\bar\gamma)\leq  \kappa$.
\end{proof}

From the above result we observe that if we can compute the distance $d_n$ explicitly, we have an explicit upper bound for the error encountered when solving \eqref{eq:min_Sn} instead of \eqref{eq:minimization}. For instance, when choosing $S_n$ as the space of the $n$-th order polynomials (see Example \ref{ex:polynomials}), an explicit error estimate can be obtained as shown next.

\begin{proposition}
\label{prop:opt real seq}
Let Assumption \ref{ass:initial}, \ref{ass:regularity-for-SDE-no-basis} and \ref{assum:extra_reg} hold true for some $N\in\mathbb{N}$, $\RR>0$. { Let $g_0,\cdots,g_n$ be a family of linearly independent vectors of $C([0,T],\bR)$ generating the space of the $\mathbb{R}^K$-valued polynomials of order $n$, i.e. $S_n = \mathcal{P}_n(\mathcal{R}^K)$ in \eqref{eq:Sn}.} 
Then the minimization problem \eqref{eq:min_Sn} admits a solution. 
Furthermore, for any 
\begin{equation}
\bar a_n \in \argmin_{a\in \mathbb{R}^{(n+1) K}} G(a)
\end{equation}
we have
\begin{equation}\label{eq:estimate_domain_red_poly}
F(\Lc\bar a_n)   \leq \frac{C}{n^N}, 
\end{equation}
where $C>0$ depends on $T,N,\RR$ 
and the dimensions (and is independent of $n$). 
\end{proposition} 
\begin{proof} 
By Lemma \ref{lemm:err_reduction} the minimization problem \eqref{eq:min_Sn} admits a solution and \eqref{eq:estimate_domain_red} holds true.  
By Proposition \ref{prop:regularity_gamma}, $\bar\gamma$ is $N$-times differentiable on $[0,T]$. In particular, \eqref{eq:represent_deriv_gamma} implies that its derivatives are bounded on $[0,T]$ by a constant that only depends on  $N$, $\RR$ and the dimensions.  

A theorem of D.~Jackson (see \cite[Section 5.1.5 (p.261)]{timan1994theory}) then yields
\begin{equation}
d_n(\bar\gamma)  
 \leq  \frac{C}{n^N} ,
\end{equation}
which concludes the proof. 
\end{proof}

{\begin{remark}\label{remark:LiftingBoundedness}
We stress that the result of Proposition \ref{prop:opt real seq} does not depend on the choice of the basis $g_i$, $i=0,\dots,n$. In particular, it applies to all the possible choices as discussed in Remark \ref{ex:polynomials}. We claim that the results of this section would still hold true after lifting the boundedness assumption on $\alpha$, $\beta$ and $\varphi$. Indeed, in light of Remark \ref{rem:sublin}, we expect \eqref{eq:minimization} to admit a unique solution $\bar\gamma\in C([0,T],\bR^K)$ even if $\alpha(t,\cdot)$, $\beta(t,\cdot)$ have sub-linear growth and $\varphi$ has polynomial growth (provided that the initial datum $X_0=\xi$ has suitably many finite moments). Furthermore, the proof of Lemma \ref{lemm:err_reduction} only relies on the estimates of Lemma \ref{lem:F_lipschitz}, whose proof itself would be identical in this more general context. 
Similarly, by Remark \ref{rem:ext_reg_der}, we expect Proposition \ref{prop:opt real seq} to still hold true if the derivatives of $\alpha(t,\cdot)$, $\beta(t,\cdot)$ and $\varphi$ have polynomial growth.
\end{remark}}

\section{Stochastic gradient descent for MV-SDEs}
\label{sec:sgd for mvsde}
We propose a stochastic gradient descent algorithm (SGD) to solve the minimization problem \eqref{eq:min_Sn}.  
For the convenience of the reader we re-write the objective function $G$ in \eqref{eq:min_Sn}  explicitly as
\begin{equation}\label{eq:FunctionalG}
\mathbb{R}^{(n+1)K} \ni a \mapsto 
G (a) 
= \int_0^T  \big| \Eb \big[ \varphi(Z^{a}_t) - \h\big((\Lc a) (t)\big) \big] \big|^2   \dd t + H(a), 
\end{equation}
where we use the shorthand notation $Z^{a}_t := Z^{\h\circ \Lc a}_t$ to denote the solution to \eqref{eq:SDE_Z} when $\gamma$ is taken as $\gamma=\h\circ \Lc a$ , i.e., $Z^{a}$ takes values in $\Rb^d$ and solves 
\begin{equation}\label{eq:SDE_Z_ter}
    \dd Z_t = \h\big( (\Lc a)(t)\big) \big( \alpha(t,Z_t)  \dd t +   \beta(t,Z_t) \dd W_t \big),\qquad Z_0=\xi.
\end{equation}
We point out that the objective function $G$ only depends on the law of $Z^{a}$ (and $a$). Thus, when weak uniqueness holds (for instance under Assumption \ref{ass:initial} and \ref{ass:regularity-for-SDE-no-basis}),  
$G$ can be defined in terms of any weak solution to \eqref{eq:SDE_Z_ter}. In particular, it only depends on the coefficients of \eqref{eq:SDE_Z_ter} and on the law of $\xi$, but not on the specific realizations of $\xi$ or the Brownian motion $W$. 
However, in the sequel, we will sometimes consider strong solutions associated to specific realizations of Brownian motion and initial datum. In these cases we will reinforce the notation by denoting $Z^{a}(\xi ,  W)$ as the strong solution to \eqref{eq:SDE_Z_ter} with respect to a given initial random variable $\xi$ and a given Brownian motion $W$.

Lastly, it is clear that within our setting,  $\min_{a\in \mathbb{R}^{(n+1) K}} G(a)$
is not a convex minimization problem. 

\subsection{Some preliminaries} 
To explain the stochastic gradient descent algorithm we begin with some heuristics. Hereafter we write a generic element of $a\in\mathbb{R}^{(n+1)K}$ as
\begin{equation}
a=(a_{0,1}, \cdots, a_{0,K},\cdots,a_{n,1}, \cdots, a_{n,K} ).
\end{equation}
\begin{itemize}
\item For any $k = 0, \cdots, n,$ and $j = 1,\cdots, K$, by formally differentiating \eqref{eq:SDE_Z_ter} with respect to the element $a_{k,j}$, 
we obtain 
that the pair $(Z^a,\partial_{a_{k,j}} Z^a)$ solves the system given by \eqref{eq:SDE_Z_ter} and 
\begin{align}
\dd Y^{k,j}_t  &= g_k(t) \nabla \h_j\big( \Lc a(t) \big)      \Big( \alpha(t,  Z_t ) \dd t  +  \beta(t,  Z_t ) \dd  W_t \Big) \\
&+ \sum^d_{i=1}  
Y^{k,j,i}_t
 \h\big( \Lc a(t) \big) \Big(  \partial_{z_i} \alpha(t,   Z_t )   \dd t +  \partial_{z_i} \beta (t,  Z_t)   \dd  W_t  \Big),  \qquad Y^{k,j}_0 = 0
, \label{eq:derivative_Z}
\end{align}
with $\h$ introduced in \eqref{eq:h}. 
Note that $Y^{k,j}$ and $Z$ take values in $\Rb^d$, with components denoted by $Y^{k,j,i}$ and $Z^i$, for $i=1,\cdots,d$, respectively. 
In the sequel, we will sometimes reinforce the notation by denoting $(Z^{a}_t (\xi,W),Y^{a;k,j}_t (\xi,W))$ as the strong solution to \eqref{eq:SDE_Z_ter}-\eqref{eq:derivative_Z} with respect to a given initial random variable $\xi$ and Brownian motion $W$.

\item Suppose that we can always exchange derivatives with time-integrals and expected values (we justify this exchange of operations below). For any $a\in\Rb^{(n+1) K}$, and for any $k=0,\cdots,n$ and $j=1,\cdots, K$,  we obtain
\begin{align}
\label{eq:formal_comp_bis}
\nonumber
&\partial_{a_{k,j}}  \bigg( \int_0^T  \big| \Eb \big[ \varphi(Z^{a}_t) -  \h\big((\Lc a) (t)\big) \big] \big|^2   \dd t   \bigg) \\
&= 2 \int_0^T  \Big\langle \Eb\Big[ \varphi(Z^{a}_t) -  \h\big((\Lc a) (t)\big) \Big] ,     \Eb \Big[ \partial_{a_{k,j}} \Big(  \varphi(Z^{a}_t) -  \h\big((\Lc a) (t)\big) \Big)  \Big]     \Big\rangle  \, \dd t .
\end{align}

Let now $W$, $\tilde W$ be two independent Brownian motions and $\xi$, $\tilde \xi$ be two identically distributed independent random variables, which are also independent of $(W,\tilde W)$. Assuming weak-uniqueness for the solutions to \eqref{eq:SDE_Z_ter} 
 we have
\begin{align}
&\Big\langle \Eb\Big[ \varphi(Z^{ a}_t)- \h\big((\Lc a) (t)\big) \Big] ,     \Eb \Big[ \partial_{a_{k,j}} \Big( \varphi(Z^{a}_t) - \h\big((\Lc a) (t)\big) \Big)  \Big]     \Big\rangle \\
 &= \Big\langle \Eb\Big[ \varphi \big(Z^{a}_{t}(\xi, W)\big) - \h\big((\Lc a) (t)\big)\Big] , \Eb \Big[    \partial_{a_{k,j}}  \Big(  \varphi \big(Z^{a}_{t}(\tilde\xi,\tilde W)\big)  - \h\big((\Lc a) (t)\big)  \Big)  \Big]     \Big\rangle \\
&= \Eb\Big[ \Big\langle  \varphi \big(Z^{a}_{t}(\xi, W)\big) - \h\big((\Lc a) (t)\big)  ,    \partial_{a_{k,j}}  \Big(  \varphi \big(Z^{a}_{t}(\tilde\xi,\tilde W)\big)  - \h\big((\Lc a) (t)\big)  \Big)      \Big\rangle \Big]. 
\label{eq:formal_comp}
\end{align}
\item Therefore, combining \eqref{eq:formal_comp_bis} and \eqref{eq:formal_comp}, and since $\partial_{a_{k,j}} Z^a$ solves \eqref{eq:derivative_Z}, we formally obtain 
\begin{equation}
 \nabla_a G(a)  =  
 \Eb[        v(a;  \xi, {W} ; \tilde\xi, \tilde W)   ] + \nabla_a H(a),      
\end{equation}
with $v(a;  \xi, {W} ; \tilde\xi, \tilde W)\in\mathbb{R}^{(n+1)K}$, whose components are given by
\begin{align}
\label{eq:v_function}\nonumber
&v_{k,j}(a;  \xi, {W} ; \tilde\xi, \tilde W) \\
& : = 2 \int_0^T \Big\langle  \varphi \big(Z^{a}_{t}(\xi, W)\big) - \h\big((\Lc a) (t)\big)  ,   \nabla_x \varphi \big(Z^{a}_{t}(\tilde\xi,\tilde W)\big) Y^{a;k,j}_{t}(\tilde\xi,\tilde W)   - \partial_{a_{k,j}} \h\big( (\Lc a) (t) \big)       \Big\rangle               \dd t,
\end{align}
where $\nabla_x \varphi$ stands for the Jacobian matrix of $\varphi:\Rb^d \to \Rb^K$.
\end{itemize}  

\subsection{Algorithm}
\label{sec:theALGORITHM}
In the previous section, we presented heuristics on how certain computations can be performed in order to implement the stochastic gradient algorithm. With the help of those explanations, we can now introduce the algorithm. First, let 
\begin{itemize}
\item 
$(\eta_m)_{m\in\mathbb{N}}$ to be a deterministic sequence of positive scalars (\emph{learning rates}) such that
\begin{equation}\label{eq:RM_cond_learning}
\sum_{m=0}^{\infty} \eta_m = + \infty\quad \textrm{and} \quad \sum_{m=0}^{\infty} \eta^2_m < + \infty .
\end{equation}
\item 
$( \bar{\Omega} , \bar{\mathcal{F}} , (\bar{\mathcal{F}}_m)_{m\in\mathbb{N}}  , \bar{\mathbb{P}})$ to be a filtered probability space rich enough to support four independent samples $(W_m)_{m\in\mathbb{N}}$, $(\tilde W_m)_{m\in\mathbb{N}}$, $(\xi_{m})_{m \in\mathbb{N}}$, $(\tilde\xi_{m})_{m \in\mathbb{N}}$ such that:
\begin{itemize}
\item 
$(W_m)_{m\in\mathbb{N}}$ and $(\tilde W_m)_{m\in\mathbb{N}}$ 
are sequences of independent $q$-dimensional standard Brownian motions adapted to $(\bar{\mathcal{F}}_m)_{m\in\mathbb{N}}$. Namely, for any $m\in\mathbb{N}$, $W_m = (W_{m,t})_{t\in[0,T]}$ is a $q$-dimensional standard Brownian motion such that $\mathcal{F}^{W_m}_T \subset \bar{\mathcal{F}}_m$, and the same holds for $\tilde W_m$;
\item 
$(\xi_{m})_{m \in\mathbb{N}}$, $(\tilde\xi_{m})_{m \in\mathbb{N}}$ of $\xi$ are sequences of independent $\Rb^d$-valued random variables distributed like $\xi$.
\end{itemize}
\end{itemize}
For a deterministic ${\bf a}_0 \in \Rb^{(n+1) K},$ the estimate update is obtained iteratively as 
\begin{equation}
\label{eq_def_a_v}
{\bf a}_{m+1}  := {\bf a}_m - \eta_m {\bf v}_{m+1}, \qquad {\bf v}_{m+1} = v({\bf a}_m; \xi_{m+1}, W_{m+1};\tilde\xi_{m+1}, \tilde W_{m+1}) + \nabla H({\bf a}_{m}), \qquad m\in  \mathbb{N}_0.
\end{equation}

\begin{algorithm}[H]
\caption{SGD-MVSDE}\label{alg:SGDMVSDE}
\KwIn{Maximum number of iterations $m_{\max}$, learning rates $(\eta_m)_{m\in\mathbb{N}}$\;}
\KwData{Deterministic ${\bf a}_0 \in \Rb^{(n+1) K}$\; 
$m = 0$\;}
\While{$m \leq m_{\max}$}{
Sample $(W_{m+1}, \tilde W_{m+1}, \xi_{m+1}, \tilde\xi_{m+1})$\;
	${\bf v}_{m+1} := v({\bf a}_m; \xi_{m+1}, W_{m+1};\tilde\xi_{m+1}, \tilde W_{m+1}) + \nabla H({\bf a}_{m})$\;
	${\bf a}_{m+1}  := {\bf a}_m - \eta_m {\bf v}_{m+1}$\;
}
\end{algorithm}

The function $v$ is defined component-wise by \eqref{eq:v_function}. Note that the $\Rb^{(n+1) K}$-valued stochastic processes $({\bf a}_m)_{m\in\mathbb{N}}$ and $({\bf v}_m)_{m\in\mathbb{N}}$ are recursively defined, as in Algorithm \ref{alg:SGDMVSDE}, and are adapted to the filtration $(\bar{\mathcal{F}}_m)_{m\in\mathbb{N}}$, which then represents the history of the algorithm. 

We can also formulate a mini-batch version of the SGD-MVSDE algorithm by computing the gradient of the loss function $v$ via averaging its value over $M$ replications $\bigl(W^i_{m+1},\tilde W^i_{m+1},\xi^i_{m+1},\tilde\xi^i_{m+1}\bigr)^M_{i=1}$ of $(W_{m+1},\tilde W_{m+1},\xi_{m+1},\tilde\xi_{m+1})$ for a given iteration index $m.$ 

\begin{algorithm}[H]
\caption{minibatch-SGD-MVSDE}\label{alg:mbSGDMVSDE}
\KwIn{Maximum number of iterations $m_{\max}$;  learning rates $(\eta_m)_{m\in\mathbb{N}}$; Mini-batch sample size $M$ }
\KwData{Deterministic ${\bf a}_0 \in \Rb^{(n+1) K}$\; 
$m = 0$\;}
\While{$m \leq m_{\max}$}{
Sample $\bigl(W^i_{m+1},\tilde W^i_{m+1},\xi^i_{m+1},\tilde\xi^i_{m+1}\bigr)^M_{i=1}$\;
${\bf v}_{m+1} :=\frac{1}{M} \sum^{M}_{i=1}v({\bf a}_m; \xi^i_{m+1}, W^i_{m+1};\tilde\xi^i_{m+1}, \tilde W^i_{m+1}) + \nabla H({\bf a}_{m})$\;
${\bf a}_{m+1}  := {\bf a}_m - \eta_m {\bf v}_{m+1}$\;
}
\end{algorithm}
Note that Algorithm \ref{alg:mbSGDMVSDE} for $M=1 $ is equivalent to Algorithm \ref{alg:SGDMVSDE}.
 
{\blue
\begin{figure}
{\small
\begin{center}
    \begin{tikzpicture}[node distance=1cm] 
        \node (start) [startstop] {Start};
        \node (init) [process, below of=start, node distance=1.25cm]  {Initialize \\ $\textbf{a}_0\in \mathbb{R}^{(n+1)K}$, $m=0$ };
        \node (while) [decision, below of=init, node distance=2.5cm] {While $m \leq m_{max}$ };
        \node (grad) [process, right of=while, xshift=6.25cm, yshift=0cm] {Compute the gradient \\ $\textbf{v}_{m+1} = v(\textbf{a}_m, W_{m+1}, \tilde{W}_{m+1}) + \nabla H(\textbf{a}_m)$ };
        \node (up) [process, below of=grad, node distance=1.25cm] {Update \\ $\textbf{a}_{m+1} = \textbf{a}_m - \eta_m \textbf{v}_{m+1}$ };
        \node (incr) [process, below of=up, node distance=1.25cm] {Increment \\ $m \leftarrow m + 1$ };
        \node (out) [process, below of=
        while, node distance=2cm] {Return the output \\ $\textbf{a}_m$ };
        \node (end) [startstop, below of=out, node distance=1.25cm] {Stop};

       \draw [arrow] (start) -- (init);
        \draw [arrow] (init) -- (while);
        \draw [arrow] (while.east) -- (grad.west);
        \draw [arrow] (grad) -- (up);
        \draw [arrow] (up) -- (incr);
        \draw [arrow] ([xshift=-0.5cm]incr.west) |- ([yshift=0.5cm]while.north);
        \draw  ([xshift=-0.5cm]incr.west) -- (incr.west);
       \draw [arrow] (while) -- (out);
        \draw [arrow] (out) -- (end);
    \end{tikzpicture}
\end{center}}\caption{Flow-chart of Algorithm 1}
\label{fig.FlowChart}
\end{figure}}
\subsection{Convergence of Algorithm \ref{alg:SGDMVSDE}}
\label{sec:convergence} 

We recall that $H$ is a function in $C^1(\mathbb{R}^{(n+1) K})$ and that it satisfies \eqref{eq:Hn}. In order to obtain a convergence result for Algorithm \ref{alg:SGDMVSDE}, we introduce the following assumption.
\begin{assumption}
\label{ass:H}
The function $H$ is such that $\nabla_a H$ is globally Lipschitz continuous on $\mathbb{R}^{(n+1) K}$. 
\end{assumption}
For example, the $H$ specified in Example \ref{ex:H} satisfies Assumption \ref{ass:H}. {This requirement of Lipschitz regularity is essential to prove convergence of any iterative method with a sequential step {(see Proposition \ref{lem:regul_grad_G} below).}

Our main convergence result shows the convergence of the iterates ${\bf a}_{m}$ in Algorithm \ref{alg:SGDMVSDE} to a stationary point of $G$. 

\begin{theorem}\label{th:convergence}
{ Let Assumptions \ref{ass:initial} and \ref{ass:H} be in force. Assume also that $\alpha(t,\cdot),\beta(t,\cdot)\in C^2(\mathbb{R}^d)$, for any $t\in[0,T]$, and $\varphi\in C^2(\mathbb{R}^K)$, with derivatives bounded by some $R>0$.} The map $G$ is differentiable and
 there exists an $\Rb^{(n+1)K}$-valued random variable ${\bf a}_{\infty}$ such that, in Algorithm \ref{alg:SGDMVSDE}, we have
\begin{equation}
\lim_{m\to + \infty} {\bf a}_{m} = {\bf a}_{\infty} \quad \textrm{$\bar\Pb$-almost surely}
\quad \textrm{and
} \quad \nabla_a G({\bf a}_{\infty}) = 0.
\end{equation}
\end{theorem}

To prove Theorem \ref{th:convergence}, we first prove two intermediate results. These results enable us to employ the developments in \cite{bertsekas2000gradient} and \cite{patel2021stochastic} for the convergence of SGD algorithms in non-convex settings. 
The first preliminary result contains: (i) a standard consistency condition that allows us to understand the process $({\bf v}_{m})_{m\in\mathbb{N}}$ as a sequence of stochastic gradients of $G$; and (ii) a bound on the second conditional moment of the gradient noise. 
\begin{proposition}\label{lem:consistency_noise}
{ 
Under the assumptions of Theorem \ref{th:convergence}, we have}
\begin{equation}\label{eq:consistency_condition}
\Eb_{\bar{\Pb}}\big[ {\bf v}_{m+1} - \nabla_a G( {\bf a}_m) | \bar{\mathcal{F}}_m \big] = 0, \qquad m \in \mathbb{N}_0,
\end{equation}
where we set $\bar{\mathcal{F}}_0:=\{ \bar\Omega, \emptyset  \}$. Furthermore, 
the random variables 
\begin{equation}
\Eb_{\bar{\Pb}}\big[ | {\bf v}_{m+1} - \nabla_a G( {\bf a}_m)  |^2 \big| \bar{\mathcal{F}}_m \big]  
\end{equation}
are bounded, uniformly with respect to $m\in  \mathbb{N}_0$.
\end{proposition}

The next proposition is a regularity result for $G$ and its gradient.
\begin{proposition}[Regularity of $G$ and $\nabla_a G$]\label{lem:regul_grad_G}
{ 
Under the assumptions of Theorem \ref{th:convergence}, } 
the function $G$ is continuously differentiable and its gradient $\nabla_a G$ is globally Lipschitz continuous on $\mathbb{R}^{(n+1) K}$. 
\end{proposition} 

We are now in the position to prove our main result in Theorem \ref{th:convergence}.
\begin{proof}[Proof of Theorem \ref{th:convergence}] 
Differentiability of $G$ follows from Proposition \ref{lem:regul_grad_G}. 
By Proposition \ref{lem:consistency_noise} and Proposition \ref{lem:regul_grad_G}, the hypothesis of \cite[Proposition 3]{bertsekas2000gradient} and \cite[Theorem 1]{patel2021stochastic} are fulfilled.  In particular, Theorem 1 in \cite{patel2021stochastic} (see the remark on p.5 there) implies that $\bar\Pb$-almost surely either ${{\bf a}_m}$ converges in $\mathbb{R}^{(n+1) K}$ or $|{{\bf a}_m}|\to +\infty$. On the other hand, by Proposition 3 in \cite{bertsekas2000gradient}, either $G({\bf a}_m)\to -\infty$ or $G({\bf a}_m)$ converges to a finite value and $\nabla G({\bf a}_m)\to 0$, $\bar\Pb$-almost surely. Since 
\begin{equation}
G\geq 0 \quad\textrm{and}\quad \lim_{|a|\to+\infty} G(a) = +\infty,
\end{equation}
we have that ${\bf a}_m$ converges $\bar\Pb$-almost surely to a stationary point of $G$. 
\end{proof}

{ \begin{remark} \label{remark:OnConvergenceAnalyis}
We claim that the convergence analysis of this section could be derived assuming $\alpha(t,\cdot)$, $\beta(t,\cdot)$ and $\varphi$ with sub-linear growth (instead of bounded). Indeed, Propositions 3.2 and 3.3 (on which Theorem \ref{th:convergence} relies) would be proved with similar, though slightly more involved, arguments. In particular, the estimates of Lemmas \ref{lem:estimates_z_Y} and \ref{Lemma:4.15.Diff} for the solutions to the system \eqref{eq:SDE_Z_ter}-\eqref{eq:derivative_Z}, which are the crux of the argument, should extend to the case of unbounded $\alpha(t,\cdot),\beta(t,\cdot), \varphi$ while keeping the benign boundedness assumption for their derivatives. Finally, we believe the assumptions on $\varphi$ could be further relaxed by allowing for polynomial growth, provided additional integrability conditions are assumed on the initial datum.
\end{remark}}

{\begin{remark}\label{remark:worseFunctions} 
The Lipschitz regularity of the derivatives of $\alpha(t,\cdot),\beta(t,\cdot)$ are needed to employ standard stability estimates (see Lemma A.1) to the system \eqref{eq:SDE_Z_ter}-\eqref{eq:derivative_Z}. While it is true that dropping this assumption would allow to include popular models with H\"older-continuous coefficients (e.g. square-root diffusions), it would also result in the tangent process being the solution of an SDE with singular coefficients. This would require the development of new fundamental stability estimates on the tangent SDE \eqref{eq:derivative_Z}. Therefore, this direction is left for future research.
\end{remark}}


\section[Numerical study]{Numerical study}
\label{sec:NumericalStudy}
We conduct several numerical experiments to study the numerical performance of {Algorithm \ref{alg:mbSGDMVSDE}} for different values of the mini-batch sample size $M$\footnote{The code is available at \url{https://github.com/amatoandre/Numerical-approximation-of-MKV-SDEs-via-SGD.git}}. Precisely, we test it for the following models:
\begin{enumerate}
\item The Kuramoto-Shinomoto-Sakaguchi model (Section \ref{sec:kuramoto}) which is in the form of \eqref{eq:new_MKV} with a low value of $K=3$. The goal here is to test the accuracy and efficiency of the scheme in a case where the dependence on the state is truly separated from the one on the measure. 
\item A model with a quadratic drift (Section \ref{sec:quadratic}), which is outside the scope of our Assumption \ref{ass:regularity-for-SDE-no-basis}.
Here we show, numerically, that our SGD algorithm efficiently converges, at least for small times, even if the coefficients are super-linear in the measure variable.  
\item An MV-SDE with a Gaussian convolution kernel (Section \ref{sec:convol}). The SDE has no separable coefficients, but it can be approximated by dynamics in the form of \eqref{eq:new_MKV} with a sufficiently large $K$ by employing a projection technique based on the generalized Fourier series (see \cite{belomestny2018projected}). The goal is to show that our {algorithm} can be successfully employed, in combination with the aforementioned projection technique, to obtain good approximations of the law of the original SDE.  
\end{enumerate}

The following choices are common to all our numerical experiments:
\begin{itemize}
\item The finite dimensional space $S_n$ for the domain projection of Section \ref{sec:ProjOnFiniteDimSpace} is that of the $n$-dimensional polynomials.  
In particular, for a given $n\in\mathbb{N}$ and $T >0$, the basis $g_0,\cdots,g_n$ is given by the Lagrange polynomials centred at the Chebyshev nodes {$t_0,\dots,t_n$ defined in \eqref{eq:orthog poly basis}.}  
Therefore, recalling the lifting operator $\Lc$ defined in \eqref{eq:lifting_operator}, for any $a\in \mathbb{R}^{(n+1)\times K}$ we have 
\begin{equation}
(\Lc a) (t_j)  = a_j, \qquad j=0,\cdots, n.
\end{equation} 
\item Inspired from the suggestion of learning rate in \cite{fehrman2020convergence}, the sequence of the learning rates $(\eta_m)_{m\in\mathbb{N}}$ is set as
\begin{equation}
\eta_m = r_0/(m+1)^{\rho}, \qquad r_0 > 0, \quad \rho \in (0.5,1],
\end{equation}
in compliance with the conditions \eqref{eq:RM_cond_learning}.   
\item The functions $\h$ and $H$ appearing in \eqref{eq:min_Sn} are chosen as the identity and null functions, respectively. Although this choice does not meet the assumptions of the convergence results of Section \ref{sec:convergence}, it does simplify the implementation. Also, we observe that the algorithm still converges in our experiments.
\item The initial iterate ${\bf a}_0$ (algorithm initialization) is always taken as
\begin{equation}
{\bf a}_0 = \big(\varphi(\tilde X_0), \cdots, \varphi(\tilde X_0)\big),
\end{equation}
with $\tilde X_0$ being a random realization of $X_0$, which yields $\Lc {\bf a}_0 (t) \equiv \varphi(\tilde X_0)$. In particular, when $\mu_{X_0} = \delta_x$, then ${\bf a}_0 = (\varphi(x),\cdots,\varphi(x))$ is deterministic.
\item The benchmark value for the curve $\bar\gamma_\cdot = \Eb[\varphi(X_{\cdot})]$ is given by 
\begin{equation}
\label{eq:benchmark_gamma}
\bar\gamma^{\text{MC}}_\cdot = \frac{1}{N_0} \sum_{i=1}^{N_0} \varphi(X^{N_{0},i}_{\cdot}),
\end{equation}
where $X^{N_0}$ is the $N_0$-dimensional mean-field particle system associated to \eqref{eq:new_MKV}. The relative error after $m$ iterations of our SGD algorithm is  

\begin{equation}
\label{eq:varepsilon Error norm}
\varepsilon_m 
:= 
\frac{ \| \Lc {\bf{a}}_{m} - {\bar\gamma}^{\text{MC}}\|_{L^2([0,T])}}{\|{\bar\gamma}^{\text{MC}} \|_{L^2([0,T])}}, 
\end{equation} 
with $\mathbf{a}_{m}$ being the $m$-th iterate of Algorithm \ref{alg:mbSGDMVSDE} and the $L^2({[0,T]})$-norm  
over the interval $[0,T]$ is as defined in Table \ref{tab:notation}.  
The algorithm is always halted when the error $\varepsilon_m$ as computed above is below 1\%.  

\item \label{aux:RandomPageLabel-page14-LastBulletpoint3} Algorithm \ref{alg:mbSGDMVSDE} and the mean-field particle system associated with \eqref{eq:new_MKV} make use of a time-stepping scheme to approximate the solution of the involved SDEs. In particular, the system \eqref{eq:SDE_Z_ter}-\eqref{eq:derivative_Z} is discretized by employing the Euler-Maruyama scheme with time-step $h=10^{-2}$. {Note that this further approximation is absolutely standard under the assumption of $\alpha(t,\cdot),\beta(t,\cdot)$ differentiable with bounded derivatives. The discretized mean-field particle system is written as
\begin{equation}\label{eq:mean field sys}
X^{N_0,i}_t = X^{N_0,i}_{t_k} +\Bigl( \frac{1}{N_0} \sum_{i=1}^{N_{0}} \varphi(X^{N_{0},i}_{\eta(t)})\Bigr) \cdot \Bigl( \alpha(\eta(t),X^{N_{0},i}_{\eta(t)}) (t - t_k) +\beta(\eta(t),X^{N_{0},i}_{\eta(t)}) \bigl(W^i_t - W^i_{t_k} \bigr) \Bigr),
\end{equation}
with $\eta(t):=t_k$ if $t \in [t_k, t_{k+1}),$ and $t_{k+1}-t_k = h$. Here $\{t_k\}_{k}$ is a partition of $[0,T]$. Due to interactions between discretized diffusions, implementation of the above particle system will require $O(N_{0}^2K h^{-1})$ arithmetic operations. On the other hand, in Algorithm \ref{alg:mbSGDMVSDE}, one sample of the mini-batch will involve $O(Kdqh^{-1}), O(nK^2d^2qh^{-1})$ and $O(nK^2dh^{-1})$ number of arithmetic operations for implementing discretized \eqref{eq:SDE_Z_ter}, \eqref{eq:derivative_Z}, and computing entries of $v$ given in \eqref{eq:v_function}, respectively. Thus, the total number of arithmetic operations in one iteration of Algorithm \ref{alg:mbSGDMVSDE} is $O(MnK^2d^2qh^{-1})$.}
\item Computational Setup: All numerical experiments are conducted on the computing server of the Department of Mathematics, University of Bologna, featuring a 12th Gen Intel Core i9-12900K CPU, 125 GB RAM, and running Ubuntu 24.04.1 LTS. The computations were implemented in Python 3.11.11 using NumPy.
\end{itemize}

\subsection{Kuramoto-Shinomoto-Sakaguchi MV-SDE}\label{sec:kuramoto}
Consider the following MV-SDE, which is related to the famous Kuramoto-Shinomoto-Sakaguchi model (see \cite[Section 5.3.2]{frank2005nonlinear}):
\begin{equation*}
 dX_t =  \bigl(\mathbb{E}[\sin(X_t)] \cos(X_t) - \mathbb{E}[\cos(X_t)] \sin(X_t) \bigr) \textrm{d} t + \sigma \, \textrm{d} W_t, \quad X_0 = x_0.
\end{equation*}
The above MV-SDE can be seen in the form of  \eqref{eq:new_MKV} with $K=3$, $d=1$, $q=1$ and 
\begin{align}
\varphi(x) = \bigl( \sin(x),  \cos(x), 1\bigr), 
&&\alpha(t,x) = ( \cos(x), -\sin(x), 0)^\top, 
&& \beta (t,x) =( 0, 0, \sigma )^{\top}.
\end{align}
Here, $\bar\gamma (t) = \Eb[\varphi(X_t)]$ takes values on $\Rb^3$. However, as $\bar{\gamma}_3(t)\equiv 1$, we run the SGD algorithm with $K=2$ to approximate $\bar\gamma_1(t)=\Eb[\sin(X_t)]$ and $\bar\gamma_2(t) =  \Eb[\cos(X_t)]$, which are the real unknown functions. Note that $\varphi, \alpha$ and $\beta$ in this example satisfy the assumptions required in the convergence result of Section \ref{sec:convergence}.  

We test different choices of algorithm hyperparameters: $r_0 = 1, 5, 10$, $\rho = 0.6, 0.7, 0.8, 0.9$, $n=3,4,5,6{, 7, 8}$ and $M= 1, 10, 100, 1000, 10000$. The benchmark ${\bar\gamma}^{\text{MC}}$ in \eqref{eq:benchmark_gamma}, based on the Monte Carlo particle system approximation, is obtained with $N_{0}=10^6$ particles. For each combination of hyperparameters, we perform { 1000} independent runs of the algorithm, and record the average number of iterations of the SGD algorithm required to achieve the relative error accuracy $\varepsilon_m<1\%$. 

In Tables \ref{tab:kuramoto1}, \ref{tab:kuramoto2}, and \ref{tab:kuramoto3} we report the lowest average number of iterations obtained over different $(r_0,\rho)$ combinations for a given $M$ and $n$  for time horizons $T=0.5, 1, 2,$ respectively. 
We observe that the average number of iterations 
decreases with the size of the mini-batch sample $M$, and increases with the time-horizon $T$.  
By taking into account the computational time, 
it is clear that $M=1000$ gives the best performance. It is evident that our \emph{minibatch-SGD} algorithm converges very quickly to the Monte Carlo particle system based benchmark.

In Figures \ref{fig:kuramoto1}, \ref{fig:kuramoto2}, \ref{fig:kuramoto3} we plot the output curves $\Lc {\bf a}_m=(\Lc {\bf a}_m^1, \Lc {\bf a}_m^2)$ of the SGD algorithm versus the benchmark curves ${\bar\gamma}^{\text{MC}}=({\bar\gamma}^{\text{MC}}_1 , {\bar\gamma}^{\text{MC}}_2),$ { for all values of the polynomial degree $n=3, 4, ..., 8$ and} for time horizons $T=0.5, 1, 2,$ respectively. 
Note that the approximation does not deteriorate as the time horizon $T $ increases. In particular, the slight loss of precision that can be observed towards the right end of the time interval can be explained by the choice of polynomial approximations. For instance, the gap between $(\Lc {\bf a}_m)^1$ and ${\bar\gamma}^{\text{MC}}_1$, which can be seen in Figure \ref{fig:kuramoto1} in the region $t\in[0.4,0.5]$, does not appear in the analogous plot of Figure \ref{fig:kuramoto2} (or Figure \ref{fig:kuramoto3}) after the same number of steps. 
This shows that the error does not propagate over time: this is mainly due to the fact that ours is a full history-type minimization algorithm, in that it approximates the minimum in \eqref{eq:minimization}  over the set of the continuous curves on $[0,T]$. 

{ Finally observe that the algorithm works best for low values of the polynomial order, namely $n=3,4$. Indeed, typically the number of iterations needed to reach the target relative error accuracy $\varepsilon_m<1\%$ (cf. Tables \ref{tab:kuramoto1}, \ref{tab:kuramoto2}, \ref{tab:kuramoto3}) increase as higher values of $n$ are considered. For the best values of the mini-batch size, namely $M=100,1000$, this phenomenon seems less evident. However, we can see from Figures \ref{fig:kuramoto1}, \ref{fig:kuramoto2}, \ref{fig:kuramoto3} that the quality of the approximation deteriorates when the number of iterations does not increase with $n$. This can be interpreted as an overfitting phenomenon, which is expected as the benchmark curves ${\bar\gamma}^{\text{MC}}=({\bar\gamma}^{\text{MC}}_1 , {\bar\gamma}^{\text{MC}}_2)$ are relatively flat.} 
\begin{table}[htpb]
\centering
{
\begin{tabular}{|c|lllll|}
\hline
$ $ & $M = 1$ & $M = 10$ & $M= 100$ & $M= 1000$ & $M= 10000$\\
\hline
$n= 3$ & 249.6 (0.23) & 40.4 (0.05) & 8.3 (0.01) & 2.6 (0.02) & 2 (0.17) \\
$n= 4$ & 368 (0.34) & 48 (0.06) & 8.2 (0.01) & 2.6 (0.02) & 2 (0.19) \\
$n= 5$ & 382.4 (0.36) & 59.6 (0.06) & 8.3 (0.01) & 2.4 (0.02) & 1.1 (0.14)\\
$n= 6$ & 405 (0.38) & 53.8 (0.06) & 8.3 (0.02) & 2.4 (0.02) & 1.1 (0.15)\\
$n= 7$ & 447.3 (0.42) & 53.9 (0.06) & 8.7 (0.02) & 2.7 (0.02) & 1.5 (0.22)\\
$n= 8$ & 451 (0.42) & 55.2 (0.07) & 10.3 (0.02) & 2.7 (0.03) & 2 (0.33)\\
\hline
\end{tabular} 
}
\caption{{ Kuramoto-Shinomoto-Sakaguchi MV-SDE.} Average number of iterations $m$ (execution time in seconds) over { 1000} independent runs of the algorithm to achieve relative error accuracy $\varepsilon_m<1\%$ with the best combination of $(r_0,\rho)$ {for each pair $(n,M)$}. Here $T=0.5$, $x_0=0.5$, $\sigma=0.5$ { and the Monte Carlo benchmark exhibited an execution time of 1.98 seconds.}}
\label{tab:kuramoto1}
\end{table} 

\begin{figure}[H]
\centering
\includegraphics[width=0.9\textwidth]{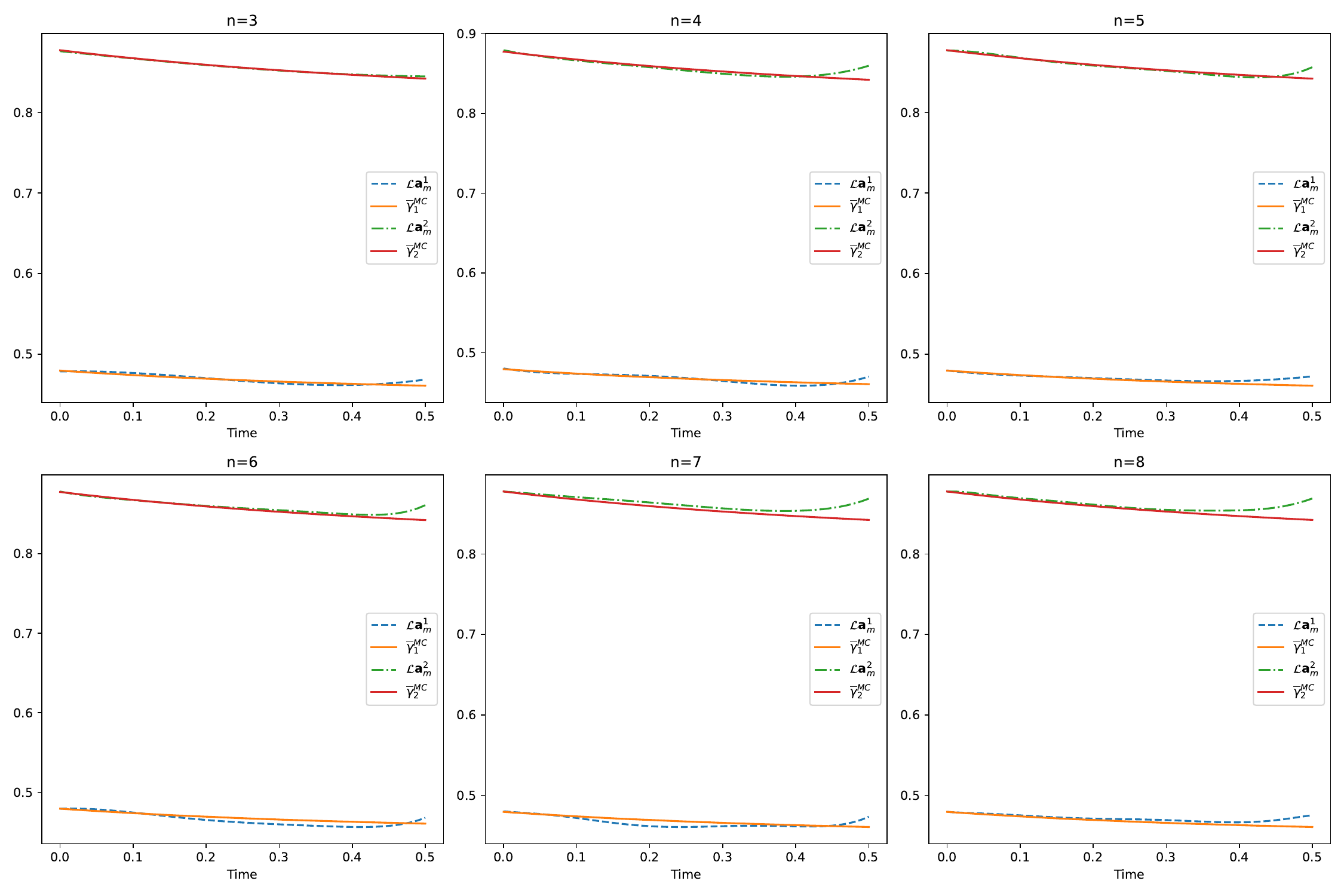}
\caption{{ Kuramoto-Shinomoto-Sakaguchi MV-SDE. Comparison of the output curves $\Lc {\bf a}_m=(\Lc {\bf a}_m^1, \Lc {\bf a}_m^2)$ of the SGD algorithm versus the benchmark curves ${\bar\gamma}^{\text{MC}}=({\bar\gamma}^{\text{MC}}_1 , {\bar\gamma}^{\text{MC}}_2)$ for all values of $n$, for timestep size $h=10^{-2}$, $T=0.5$, $M=1000$, $x_0=0.5$ and $\sigma=0.5$.}}
\label{fig:kuramoto1}
\end{figure}
 
\begin{table}[htpb]
\centering
{
\begin{tabular}{|c|lllll|}
\hline
$ $ & $M = 1$ & $M = 10$ & $M= 100$ & $M= 1000$ & $M= 10000$\\
\hline
$n= 3$ & 444.1 (0.80) & 77.5 (0.18) & 16.6 (0.06) & 5.1 (0.08) & 4 (0.65) \\
$n= 4$ & 747 (1.35) & 101.3 (0.23) & 16.1 (0.06) & 4.4 (0.07) & 3 (0.62)\\
$n= 5$ & 709.9 (1.31) & 105.8 (0.24) & 16.9 (0.07) & 4.5 (0.08) & 3 (0.71) \\
$n= 6$ & 804.3 (1.48) & 109.9 (0.25) & 17.8 (0.06) & 4.1 (0.08) & 2.1 (0.57) \\
$n= 7$ & 899.9 (1.67) & 110.9 (0.27) & 16.5 (0.06) & 4 (0.08) & 2 (0.62) \\
$n= 8$ & 923.4 (1.70) & 118.4 (0.28) & 17.9 (0.07) & 3.9 (0.09) & 2 (0.67)\\
\hline
\end{tabular} 
}
\caption{{ Kuramoto-Shinomoto-Sakaguchi MV-SDE.} Average number of iterations $m$ (execution time in seconds) over { 1000} independent runs of the algorithm to achieve relative error accuracy $\varepsilon_m<1\%$ with the best combination of $(r_0,\rho)$ {for each pair $(n,M)$}. Here $T=1$, $x_0=0.5$, $\sigma=0.5$ { and the Monte Carlo benchmark exhibited an execution time of 3.60 seconds.} }
\label{tab:kuramoto2}
\end{table} 

\begin{figure}[htpb]
\centering
\includegraphics[width=0.9\textwidth]{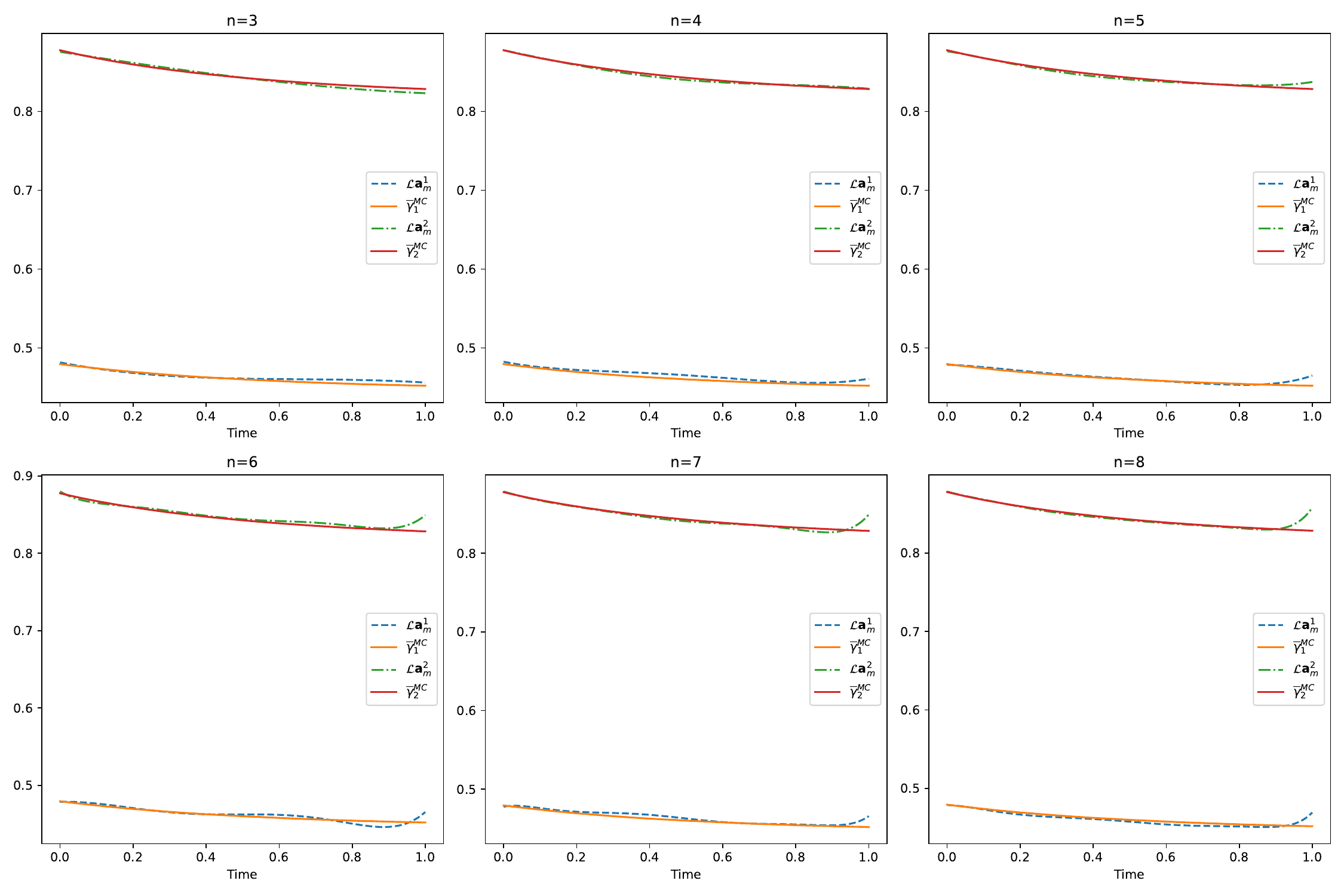}
\caption{{ Kuramoto-Shinomoto-Sakaguchi MV-SDE.} 
{ Comparison of the output curves $\Lc {\bf a}_m=(\Lc {\bf a}_m^1, \Lc {\bf a}_m^2)$ of the SGD algorithm versus the benchmark curves ${\bar\gamma}^{\text{MC}}=({\bar\gamma}^{\text{MC}}_1 , {\bar\gamma}^{\text{MC}}_2)$ for all values of $n$, for timestep size $h=10^{-2}$, $T=1$, $M=1000$, $x_0=0.5$ and $\sigma=0.5$.}}
\label{fig:kuramoto2}
\end{figure}
 
\begin{table}[htpb]
\centering
{
\begin{tabular}{|c|lllll|}
\hline
$ $ & $M = 1$ & $M = 10$ & $M= 100$ & $M= 1000$ & $M= 10000$\\
\hline
$n= 3$ & 992.1 (3.53) & 134 (0.61) & 21.4 (0.14) & 4.7 (0.15) & 2.3 (0.86) \\
$n= 4$ & 1302.7 (4.67) & 153.8 (0.70) & 25.4 (0.18) & 6 (0.21) & 3.1 (1.33)\\
$n= 5$ & 1576.5 (5.66) & 191.4 (0.87) & 28.6 (0.20) & 6.8 (0.26) & 4.1 (1.99) \\
$n= 6$ & 1730.9 (6.22) & 244.5 (1.12) & 34.9 (0.25) & 7.8 (0.33) & 5 (2.71)\\
$n= 7$ & 1950.3 (7.02) & 220.3 (1.01) & 36.1 (0.27) & 7.7 (0.35) & 4 (2.43) \\
$n= 8$ & 2033.8 (7.32) & 238.1 (1.09) & 34.5 (0.26) & 7.6 (0.37) & 4 (2.68) \\
\hline
\end{tabular} }
\caption{{ Kuramoto-Shinomoto-Sakaguchi MV-SDE.} Average number of iterations $m$ (execution time in seconds) over { 1000} independent runs of the algorithm to achieve relative error accuracy $\varepsilon_m<1\%$ with the best combination of $(r_0,\rho)$ {for each pair $(n,M)$}. Here $T=2$, $x_0=0.5$, $\sigma=0.5$ { and the Monte Carlo benchmark exhibited an execution time of 6.59 seconds.}}
\label{tab:kuramoto3}
\end{table} 

\begin{figure}[htpb]
\centering
\includegraphics[width=0.9\textwidth]{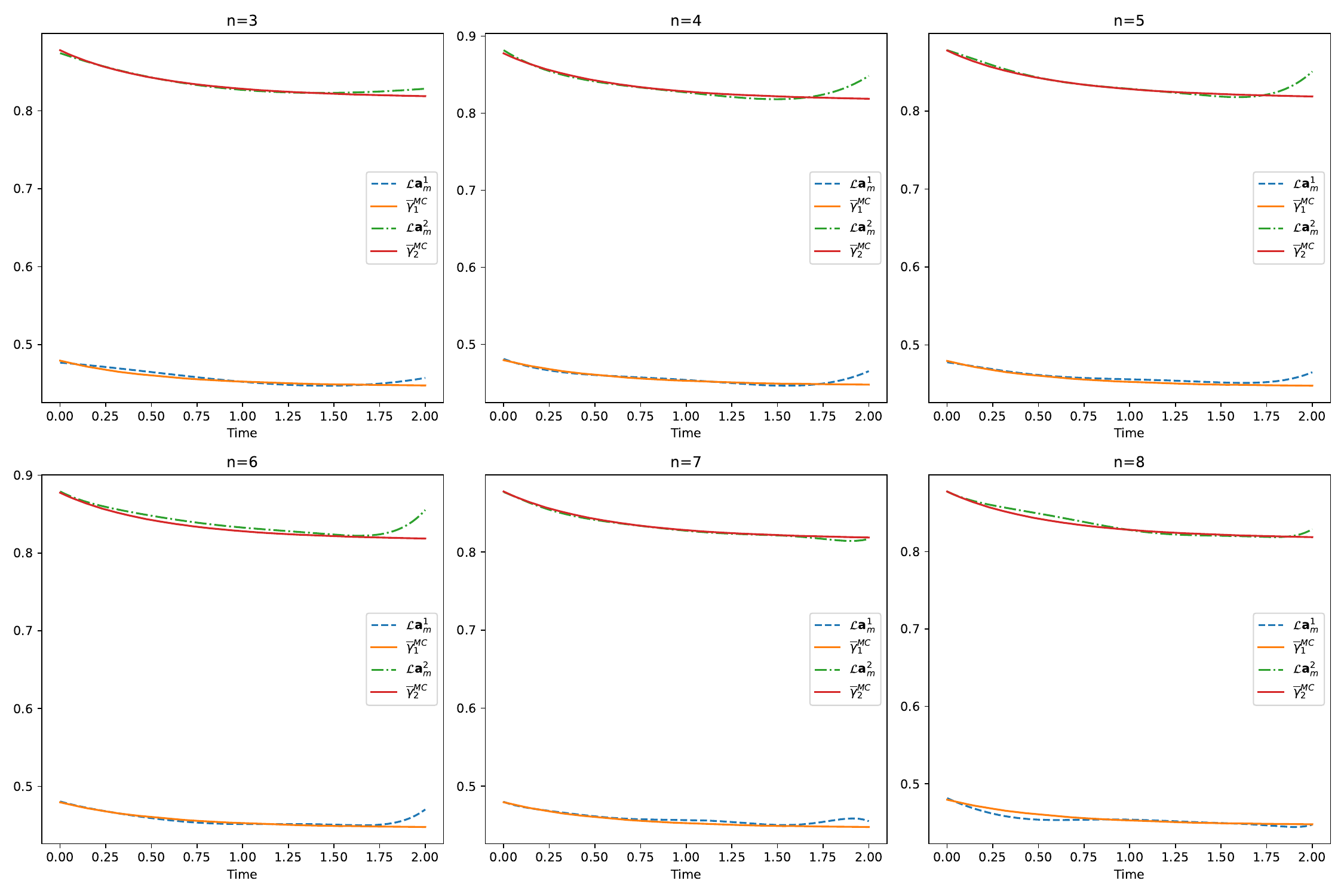}
\caption{{ Kuramoto-Shinomoto-Sakaguchi MV-SDE.} 
{ Comparison of the output curves $\Lc {\bf a}_m=(\Lc {\bf a}_m^1, \Lc {\bf a}_m^2)$ of the SGD algorithm versus the benchmark curves ${\bar\gamma}^{\text{MC}}=({\bar\gamma}^{\text{MC}}_1 , {\bar\gamma}^{\text{MC}}_2)$ for all values of $n$, for timestep size $h=10^{-2}$, $T=2$, $M=1000$, $x_0=0.5$ and $\sigma=0.5$.}}
\label{fig:kuramoto3}
\end{figure}

\subsection{Polynomial drift MV-SDE}
\label{sec:quadratic}
Consider the MV-SDE with polynomial drift given as follows:
\begin{equation}\label{eq:mkv_quadratic}
dX_t = \left( \mathbb{E}[X_t] - X_t \mathbb{E}[X_t^2] +  \delta X_t  \right) dt + X_t dW_t , \ \ \ X_0=x_0.
\end{equation} 
The latter can be cast in the form of  \eqref{eq:new_MKV} with $K=3$, $d=1$, $q=1$ and 
\begin{align}
\varphi(x)=(x, x^2, 1), && \alpha(t,x)=(1, -x, \delta x)^T,  && \beta(t,x)=(0 , 0, x)^T.
\end{align}
Once more, the function $\bar\gamma (t) = \Eb[\varphi(X_t)]$ takes values on $\Rb^3$ but $\bar{\gamma}_3(t)\equiv 1$. Thus we run the SGD algorithm with $K=2$ to approximate $\bar\gamma_1(t)=\Eb[(X_t)^2]$ and $\bar\gamma_2(t) = \Eb[X_t]$, which are the real unknown functions. Also note that $\varphi, \alpha$ and $\beta$ in this example are unbounded and thus do not satisfy the assumptions in the results of Section \ref{sec:convergence}. In particular, the drift coefficient has super-linear growth in the measure component.

We test different choices of algorithm hyperparameters: $r_0 = 1, 5, 10, \rho = 0.6, 0.7,$ and $M= 1, 10, 100, 1000$. We set the polynomial degree $n=3,$ and test the performance of the SGD algorithm for different time-horizons, $T=0.1, 0.5, 1$. The benchmark ${\bar\gamma}^{\text{MC}}$ in \eqref{eq:benchmark_gamma} based on the Monte Carlo particle system approximation is obtained with $N_{0}=10^{ 7}$ particles. For each combination of the hyperparameters, we perform {1000}  independent runs of the algorithm, and record the average number of iterations of the SGD algorithm required to achieve the relative error accuracy $\varepsilon_m<1\%$. 

In Table \ref{tab: quad1} we report the lowest average number of iterations over different $(r_0,\rho)$ combinations for a given $M$ and $T$. Once more, as the value of the mini-batch increases the number of steps to reach convergence decreases. Compared to the model tested in Section \ref{sec:kuramoto}, increasing the time-horizon $T$ has a stronger impact on the convergence, which is due to the quadratic growth of the coefficient. We stress that for this model even the MC particle method requires an increasing number of particles to obtain reliable benchmark curves for increasing values of $T$. 

\begin{table}[htpb]
\centering 
{
\begin{tabular}{|c|llll|}
\hline
$ $ & $M = 1$ & $M = 10$ & $M= 100$ & $M= 1000$ \\
\hline
$T= 0.1$ & 258 (0.06) & 52.3 (0.02) & 31.5 (0.01) & 26.9 (0.03) \\
$T= 0.5$ & 3456.7 (3.52) & 590.3 (0.72) & 131.2 (0.20) & 79.7 (0.39)\\
$T= 1.0$ & 4995.9 (9.91) & 3582 (8.54) & 1003.9 (3.07) & 776.8 (8.25)\\
\hline
\end{tabular}
}
\caption{{ Polynomial drift MV-SDE.} Average number of iterations (execution time in seconds) over { 1000} independent runs of the algorithm to achieve relative error accuracy $\varepsilon_m<1\%$ with the best combination of $(r_0,\rho)$ for {each pair ($T,M$)}.  
Here $x_0=1$, $\delta=0.8$ { and the Monte Carlo benchmark exhibited an execution time of 2.57 seconds for $T=0.1$, 9.78 seconds for $T=0.5$ and 19.53 seconds for $T=1$.}}
\label{tab: quad1}
\end{table}

In Figure \ref{fig:quadr} we plot the output curves $\Lc {\bf a}_m=(\Lc {\bf a}_m^1, \Lc {\bf a}_m^2)$ of the SGD algorithm versus the benchmark curves ${\bar\gamma}^{\text{MC}}=({\bar\gamma}^{\text{MC}}_1 , {\bar\gamma}^{\text{MC}}_2)$ for { all time horizons $T=0.1$, $T=0.5$ and $T=1$}. 
In the plots, particularly for $T=0.5$, notice the considerable loss of accuracy in our approximation $\Lc {\bf a}_m^2$ of $\bar\gamma_2(t) = \Eb[X_t]$ 
near the right boundary of the interval (similar effects were visible in the plots relative to the Kuramoto model). 
After running additional tests, we do not believe that this effect is due to the polynomial projection, similar to Runge's type phenomena in polynomial interpolation over compact sets, but we claim that it is rather due to the norm used to define the cost function $F^2$ (cf. \eqref{eq:def_F}), in combination with the chosen threshold level we set for $\varepsilon_m$ in \eqref{eq:varepsilon Error norm}. Indeed, lowering this threshold from to $10^{-2}$ to $10^{-3}$ reduces considerably the error, also near the boundaries of the interval. However, one could look for strategies to distribute the error in a more uniform way across the interval, and obtain a better qualitative result while keeping the same error threshold. 
One possibility of improvement is to first employ the alternative objective function presented in Remark \ref{Rem:More weight towards T}, i.e.,  using a weight function to force the SGD algorithm to be more accurate towards the end of the interval. One can also use a similar weight function to compute a weighted $L^2({[0,T]})$ norm for error $\varepsilon_m$ in the stopping criterion. Another possibility is to experiment with different choices of basis functions listed in Example \ref{ex:polynomials} and Example \ref{ex:spline} and choose the one with best error performance.

\begin{figure}[htpb]
\centering
\includegraphics[width=0.9\textwidth]{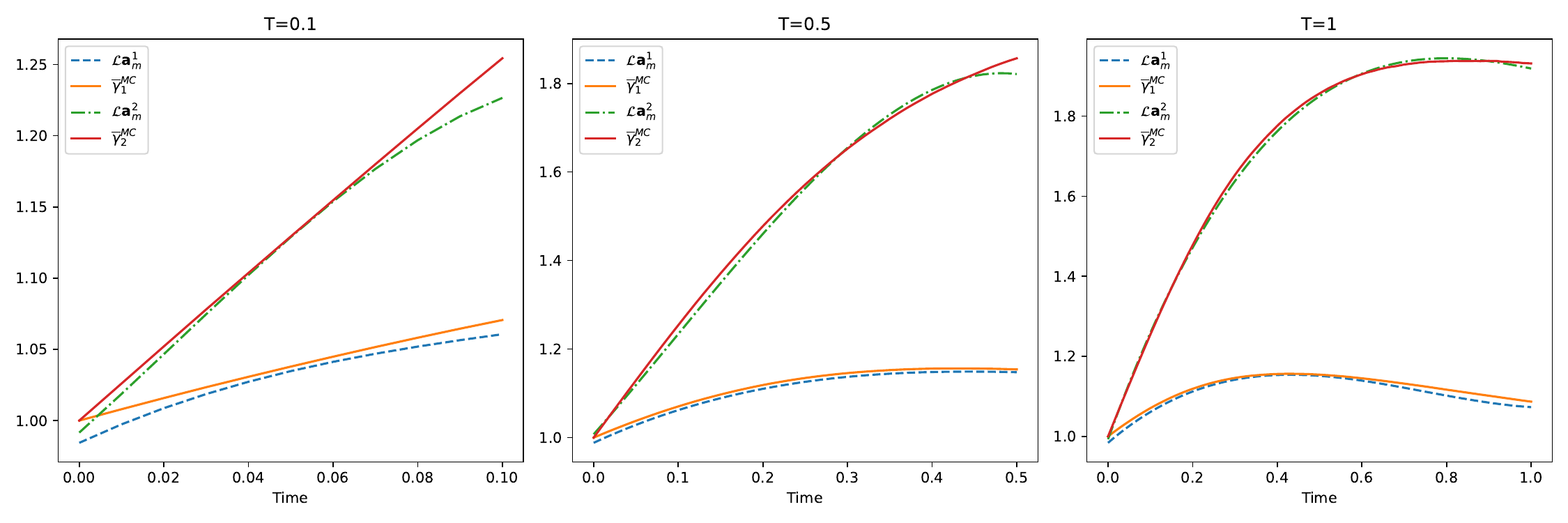}
\caption{{ Polynomial drift MV-SDE.} 
{ Comparison of the output curves $\Lc {\bf a}_m=(\Lc {\bf a}_m^1, \Lc {\bf a}_m^2)$ of the SGD algorithm versus the benchmark curves ${\bar\gamma}^{\text{MC}}=({\bar\gamma}^{\text{MC}}_1 , {\bar\gamma}^{\text{MC}}_2)$ for all values of $T$, for timestep size $h=10^{-2}$, $n=3$, $M=1000$, $x_0=1$ and $\delta=0.8$.}}
\label{fig:quadr}
\end{figure}

\subsection{Convolution-type MV-SDE}
\label{sec:convol}
We consider the following convolution-type MV-SDE:
\begin{equation}\label{eq:MKV_con}
\textrm{d} X_t = \left. \mathbb{E}\left[\exp\left(-{(X_t- x)^2}/{2}\right)\right]\right|_{x=X_t}  \textrm{d} t +  \sigma\, \textrm{d} W_t, \qquad X_0 = \xi \sim \mathcal{N}_{(0,1)}.
\end{equation}
As shown in \cite{belomestny2018projected}, the equation above can be approximated by employing a projection technique based on generalized Fourier series, which we sketch out here for the reader's convenience. Noticing that the drift coefficient in \eqref{eq:MKV_con} can be written as 
\begin{equation}
\Eb[b(x,X_t)]|_{x=X_t}, \qquad b(x,y) =  \exp\left(-{(y- x)^2}/{2}\right),
\end{equation}
we can expand $b(\cdot,x)$ in a suitable weighted $L^2$ space as follows. Setting the measure $\mu(dx) = e^{-x^2/2} dx$ on the Borel sets of $\Rb$, we consider the following orthonormal basis of $L^2(\Rb,\mu)$ of the normalised Hermite polynomials
\begin{equation}
\overline{H}_k(x) = c_k (-1)^k \mathrm{e}^{x^2} \frac{\mathrm{d}^k}{\mathrm{d} x^k} \bigl( \mathrm{e}^{-x^2}\bigr), \qquad c_k = \bigl(2^k k! \sqrt{\pi} \bigr)^{-1/2}, \qquad k\in\mathbb{N}_0.
\end{equation}
Therefore, for $K\in\mathbb{N}$ sufficiently large, $b(x,\cdot)$ can be approximated in $L^2(\Rb,\mu)$ as
\begin{equation}\label{eq:b_expans_four}
b(x,y) \approx \sum_{k=0}^{K} \alpha_k(x) \varphi_k(y),
\end{equation}
where
\begin{align}
\varphi_k(x) = \overline{H}_k(x) \mathrm{e}^{-x^2/2}, && \alpha_k(x)= \int_{\Rb} 
b(x,y) \varphi_k(y) dy 
 =  \pi^{1/4}\Bigl( \frac{1}{2}\Bigr)^{k/2} \frac{x^k}{\sqrt{k!}} \mathrm{e}^{-x^2/4}.
\end{align}
Therefore, the MV-SDE \eqref{sec:convol} can be approximated by the MV-SDE with separable coefficients
\begin{equation}
{d} X^{(K)}_t = \sum^K_{k=0}\mathbb {E}\big[\varphi_k(X^{(K)}_t)\big] \alpha_k\big(X^{(K)}_t\big) {d} t +  \sigma {d} W_t, 
\qquad X_0 = \xi = \mathcal{N}_{(0,1)},
\end{equation}
which can be clearly cast in the form of \eqref{eq:new_MKV}. For a given $K\in\mathbb{N}$, we run the SGD algorithm in order to approximate $\bar\gamma^{(K)}_k(t) = \mathbb {E}[\varphi_k(X^{(K)}_t)]$ for $k=0,\dots,K$.

For this model, we test the accuracy of the SGD algorithm by checking the quality of the density approximation of the solution $X_T$ to \eqref{sec:convol} that can be derived from the $\bar\gamma^{(K)}_k$ curves as follows (see also \cite{belomestny2018projected}). For a fixed $T>0$ denote by $w_T(x)$ and $w^{(K)}_T(x)$ the densities of $X_T$ and $X^{(K)}_T$, respectively. Following the same argument that led to \eqref{eq:b_expans_four}, for $K\in\mathbb{N}$ sufficiently large, $w_T(x)$ can be approximated in $L^2(\Rb,\mu)$ as
\begin{equation}\label{eq:den_expans_four}
w_T(x) \approx w^{(K)}_T(x) \approx \sum_{k=0}^{K}  \int_{\Rb} \varphi_k(y) w^{(K)}_T(y) dy\, \varphi_k(x) = \sum_{k=0}^{K}  \bar\gamma^{(K)}_k(T) \varphi_k(x) =: \tilde{w}^{(K)}_T(x) .
\end{equation}
Denote by $\tilde{w}^{(K),\text{SGD}}_T$ and $\tilde{w}^{(K),\text{MC}}_T$, the approximations of $\tilde{w}^{(K)}_T$ obtained by replacing the vector $\bar\gamma^{(K)}(T)$ with the output of the SGD algorithm, $\Lc {\bf a}_m(T)$, and with the particle system benchmark, ${\bar\gamma}^{\text{MC}}(T)$, respectively.

In our tests we set $\sigma = 0.1$ and $T = 1$ as in \cite{belomestny2018projected}. The number of particles to derive the benchmark ${\bar\gamma}^{\text{MC}}(T)$ is set as $N_{0}=10^7$. Note that in \cite{belomestny2018projected} the authors used $N_{0}=5\times 10^2$, which is why our density plots differ considerably from theirs. 

In Figure \ref{fig:convolution1} we plot the approximate density $\tilde{w}^{(K),\text{MC}}_T(x)$ for $K = 3, 5, 10, 20$, with $x$ ranging on the interval $[-3,4]$ . From this we conclude that the value $K=10$ is high enough in order for $\tilde{w}^{(K)}$ to accurately approximate ${w}_T$ over the interval $[-3,4]$. In Figure \ref{fig:convolution2} we test the accuracy of the SGD algorithm by plotting $\tilde{w}^{(K),\text{SGD}}_T(x)$ versus $\tilde{w}^{(K),\text{MC}}_T(x)$, with $K=10$, over the same interval. The parameters of the SGD algorithm were chosen as $n = 3$, $M = 100$, $r_0 = 5$, $\rho = 0.9$, and the algorithm halted after $m=172$ iterations ($35$ seconds of computation time) when the relative error reached $\varepsilon_m<1\%$.  
\begin{figure}[htpb]
\centering
\includegraphics[width=0.6\textwidth]{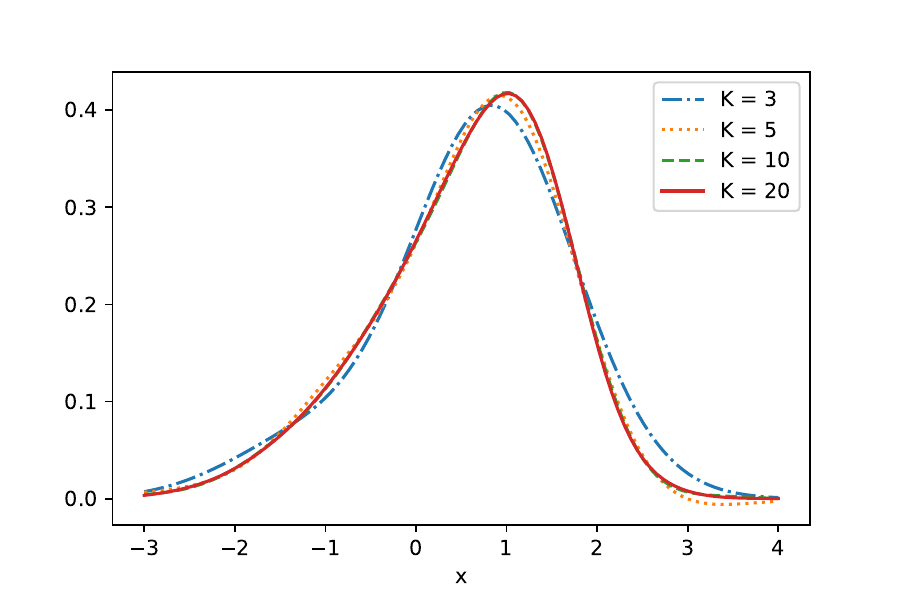}
\caption{Monte Carlo method densities $\tilde{w}^{(K),\text{MC}}_T(x)$ over the interval $[-3,4]$, for $T=1$ and $K = 3, 5, 10, 20$. The benchmark vector ${\bar\gamma}^{\text{MC}}(T)$ was computed with $N_{0}=10^7$ particles. Here $X_0 \sim \mathcal{N}_{(0,1)}$ and $\sigma = 0.1$. In the model \eqref{eq:MKV_con} we had $X_0 \sim \mathcal{N}_{(0,1)}$ and $\sigma = 0.1$.}
\label{fig:convolution1}
\end{figure}

\begin{figure}[htbp]
\centering
\includegraphics[width=0.6\textwidth]{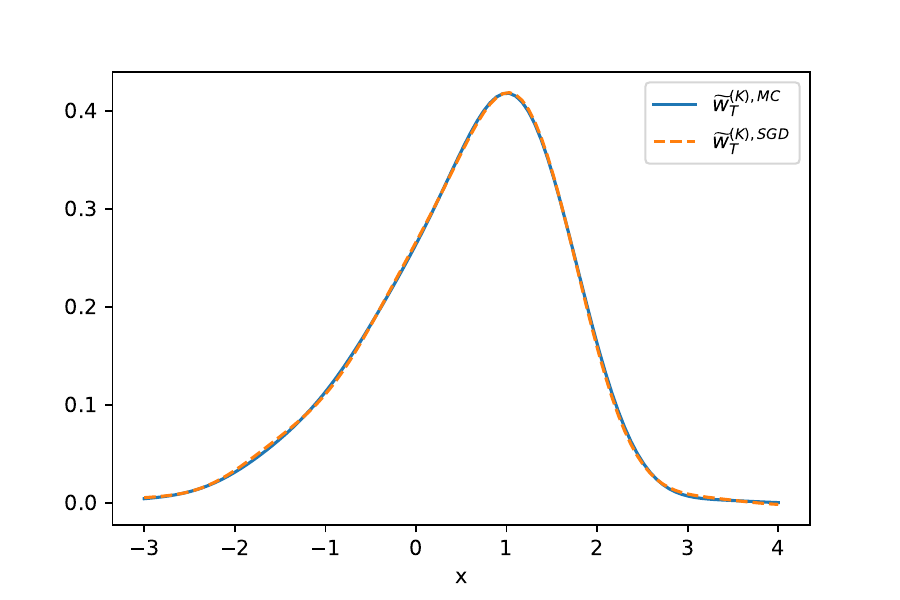}
\caption{Comparison of $\tilde{w}^{(K),\text{SGD}}_T(x)$ and $\tilde{w}^{(K),\text{MC}}_T(x)$ over the interval $[-3,4]$, for $T=1$ and $K = 10$. The benchmark vector ${\bar\gamma}^{\text{MC}}(T)$ was computed with $N_{0}=10^7$ particles. The SGD output $\Lc {\bf a}_m(T)$ was obtained after $m=172$ iterations ($35$ seconds of computation time), with parameters $n = 3$, $M = 100$, $r_0 = 5$, $\rho = 0.9$.
In the model \eqref{eq:MKV_con} we had $X_0 \sim \mathcal{N}_{(0,1)}$ and $\sigma = 0.1$.}
\label{fig:convolution2}
\end{figure}



\appendix

\section{SDE stability estimates}

Throughout this section, fix $T>0$ and let $({\Omega}, {\mathcal{F}}, ( {\mathcal{F}}_t )_{t \in [0,T]},  \Pb)$ be a filtered probability space, with $( {\mathcal{F}}_t)_{0\leq t \leq T}$ satisfying the usual hypothesis, supporting a $q$-dimensional Brownian motion ${W} = ( W_t)_{t \in [0,T]}$. Consider the $\Rb^d$-valued SDE
\begin{equation}\label{eq:sde_appendix_krylov}
\dd \cY_t = \rho_t \big( \mu(t,\cY_t)  \dd t +   \sigma(t,\cY_t) \dd W_t \big),\qquad \cY_0=\zeta,
\end{equation}
where $\zeta$ is an $\Rb^d$-valued $\mathcal{F}_0$-measurable random variable, $\mu: [0,T] \times \Rb^d \to \mathbb{R}^{m\times d}$ and $\beta: [0,T] \times \Rb^d  \to \mathbb{R}^{m \times d \times q}$ are deterministic measurable coefficients, and $\rho=(\rho_t)_{t\in[0,T]}$ is a progressively measurable $\Rb^m$-valued stochastic process. Hereafter, we denote by $\cY^{\rho}$ a solution to \eqref{eq:sde_appendix_krylov}.
\begin{lemma}
\label{lem:sde_stability_general}
Assume that $\zeta\in L^2(\Omega,\mathcal{F}_0,\Pb)$, $\rho$ is bounded, and there exists a constant $\Lambda>0$ such that 
\begin{equation}\label{eq:assump_lip_app}
|\mu(t,x)-\mu(t,y)| + |\sigma(t,x)-\sigma(t,y)| \leq \Lambda  |x-y| , \qquad t\in[0,T], \ x,y\in\Rb^d, 
\end{equation}
and
\begin{equation}\label{eq:assump_linear_growth}
|\mu(t,x)| + |\sigma(t,x)| \leq \Lambda (1 + |x|),  \qquad t\in[0,T], \ x\in\Rb^d.
\end{equation}
Then \eqref{eq:sde_appendix_krylov} has a unique solution $\cY^{\rho}$ in the sense of indistinguishability. 

Furthermore, for a given $\kappa>0$ and for any $\rho$ and $\hat\rho$ with 
$\|\rho\|_{\cS^\infty_{[0,T]}},\|\hat\rho\|_{\cS^\infty_{[0,T]}}\leq\kappa$ we have:
\begin{itemize} 
\item[(a)] if $\zeta\in L^{2p}(\Omega,\mathcal{F}_0,\Pb)$ for some $p\geq 1$, then
\begin{equation}
 \lVert \cY^{\rho}\rVert_{\cS^{2p}_{[0,T]}}   \leq C_1   \label{eq:estim_krylov} 
    \end{equation}
 where  $C_1>0$ only depends on $\kappa, p, T,\Lambda$, $\Eb [ |\zeta|^{2p} ]$ and the dimensions;
    \item[(b)] if $\zeta\in L^{4p}(\Omega,\mathcal{F}_0,\Pb)$ for some $p\geq 1$, then
\begin{equation}
   \label{eq:estim_krylov_bis}
   \lVert \cY^{\rho} - \cY^{\hat\rho} \rVert_{\cS^{2p}_{[0,T]}}     
   \leq C_2 \,           \|\rho - \hat\rho\|_{\mathbb{L}^{4p}_{[0,T]}}, 
    \end{equation}
where  $C_2>0$ only depends on $\kappa, p, T,\Lambda$, $\Eb [ |\zeta|^{4p} ]$ and the dimensions;
    \item[(c)] if 
    \begin{equation}\label{eq:bounded}
|\mu(t,x)| + |\sigma(t,x)| \leq \Lambda ,  \qquad t\in[0,T], \ x\in\Rb^d,
\end{equation}
then for any $p\geq 1$ we have
\begin{equation}
   \label{eq:estim_krylov_bis_bis}
   \lVert \cY^{\rho} - \cY^{\hat\rho} \rVert_{\cS^{2p}_{[0,T]}}     
   \leq C_3 \,           \|\rho - \hat\rho\|_{\mathbb{L}^{2p}_{[0,T]}} \leq C_3 \,  T^{\frac{1}{2p}}         \|\rho - \hat\rho\|_{\cS^\infty_{[0,T]}},
    \end{equation}
  where  $C_3>0$ only depends on $\kappa, p, T,\Lambda$ and the dimensions. 
    Note that \eqref{eq:estim_krylov_bis_bis} is true only assuming $\xi\in L^{2}(\Omega,\mathcal{F}_0,\Pb)$ and that $C_3$ does not depend on $\Eb[|\zeta|^2]$.
\end{itemize}
\end{lemma}
\begin{proof}
The well-posedness of $\cY^{\rho}$ and part (a) of the result are particular cases of \cite[Theorem 2.5.7 (p.77)]{krylov2008controlled} and \cite[Theorem 2.5.10 (p.85)]{krylov2008controlled}, respectively. 

Parts (b) and (c) can be readily obtained by combining \eqref{eq:estim_krylov} with \cite[Theorem 2.5.9 (p.83)]{krylov2008controlled}. Precisely, from \cite[Theorem 2.5.9]{krylov2008controlled}, there exists $C>0$, only dependent on $\kappa, p, T,\Lambda$ and $K$, such that
\begin{align}
\Eb\Big[ \sup_{t \leq T} |\cY^\rho_t - \cY^{\hat\rho}_t |^{2p}  \Big] &\leq C\, \Eb\bigg[ \int^T_0 \Big( \big|\rho_t \mu(t,\cY^{\hat\rho}_t) - \hat\rho_t\mu(t,\cY^{\hat\rho}_t)\big|^{2p} 
 + \big|\rho_t \sigma(t,\cY^{\hat\rho}_t) - \hat\rho_t\sigma(t,\cY^{\hat\rho}_t)\big|^{2p}  \Big) \dd t \bigg]
\\
&
\leq C\, \Eb\bigg[ \int^T_0 |\rho_t  - \hat\rho_t|^{2p} \Big( \big| \mu(t,\cY^{\hat\rho}_t)\big|^{2p}+ \big| \sigma(t,\cY^{\hat\rho}_t)\big|^{2p}   \Big) \dd t \bigg]
\\
&\leq C \Lambda\,  \|\rho - \hat\rho\|^{2p}_{\mathbb{L}^{4p}_{[0,T]}} \, \bigg( \Eb\bigg[ \int^T_0  \big( 1 + \big| \cY^{\hat\rho}_t\big|   \big)^{4p} \dd t\bigg] \bigg)^{\frac{1}{2}},
\end{align}
where the last inequality follows from \eqref{eq:assump_linear_growth} and the Cauchy-Schwarz inequality. Combining the result with \eqref{eq:estim_krylov}, yields \eqref{eq:estim_krylov_bis}. When \eqref{eq:bounded} holds true, we have a stronger estimate
\begin{equation}
\Eb\bigg[ \int^T_0 |\rho_t  - \hat\rho_t|^{2p} \Big( \big| \mu(t,\cY^{\hat\rho}_t)\big|^{2p}+ \big| \sigma(t,\cY^{\hat\rho}_t)\big|^{2p}   \Big) \dd t \bigg] \leq  \Lambda^{2p} \, \|\rho - \hat\rho\|^{2p}_{\mathbb{L}^{2p}_{[0,T]}},
\end{equation}
which yields \eqref{eq:estim_krylov_bis_bis}.
\end{proof}



\section{Proofs}
\label{sec:proofs}

\subsection[Proof of time-regularity properties]{Proof of the time-regularity of $\bar\gamma$}
\label{sec:proof time reg}
\begin{proof}[Proof of Proposition \ref{prop:regularity_gamma}]
Let $X$ be the unique solution to \eqref{eq:new_MKV}, which exists due to Proposition \ref{prop:well-posedness-of-X}. 
A straightforward application of It\^o's formula shows that its time-marginal laws $\mu_t(\dd x)$ solve the Fokker-Plank equation in distributional form over $[0, T] \times \mathbb{R}^d$, which is 
\begin{equation}
\int_0^{T} \int_{\mathbb{R}^d}   \Kc \psi(s,x)  \mu_s(\dd x) \dd s = 0, \qquad \psi\in C^{\infty}_0\big( [0, T] \times \mathbb{R}^d \big).
\end{equation}
By Assumption \ref{ass:regularity-for-SDE-no-basis} and Proposition \ref{prop:well-posedness-of-X}, the coefficients of $\Kc$ are continuous. Therefore, a standard approximating procedure yields
\begin{equation}\label{eq:deriv_expect}
\frac{\dd}{\dd t} \Eb[ f( t , X_t ) ] = \Eb[ \Kc f ( t,  X_t ) ], \qquad t\in[0,T],
\end{equation}
for any test function $f \in C^{1,2}_b ([0,T] \times \mathbb{R}^d)$. 
Now, by Assumption \ref{assum:extra_reg}, $\varphi\in C^2_b(\mathbb{R}^d)$. Thus, \eqref{eq:deriv_expect} with $f(t,x)= \varphi(x)$ becomes
\begin{equation}
\frac{\dd}{\dd t} \bar\gamma(t) = \Eb[ \Kc \varphi ( X_t ) ], \qquad t \in [0,T],
\end{equation}
which is \eqref{eq:represent_deriv_gamma} with $n=1$, and in particular $\bar\gamma\in C^{1}_b([0,T])$. 

Let now $m\in \{ 1,\cdots, N-1\}$, assume $\bar\gamma\in C^{m}_b([0,T])$ and \eqref{eq:represent_deriv_gamma} be true for $m$
. By Assumption \ref{assum:extra_reg} we obtain 
\begin{equation}
\Kc^{m} \varphi 
\in C^{1,2}_b ([0,T] \times \mathbb{R}^d).
\end{equation}
Hence, \eqref{eq:deriv_expect} with $f= \Kc^{m} \varphi $ yields
\begin{equation}
\frac{\dd^{ m +1}}{\dd t^{ m +1}} \bar\gamma(t) = \frac{\dd}{\dd t} \Eb[  \big(\Kc^{m}\varphi\big)(t,X_t) ]  = \Eb[  \big(\Kc^{m +1}\varphi\big)(t,X_t) ], \qquad t \in [0,T],
\end{equation}
and, in particular, $\bar\gamma\in C^{m+1}_b([0,T])$. This concludes the proof.
\end{proof}

\subsection{Auxiliary stability and moment estimates results}

Hereafter, we denote by $(Z^a,Y^{a;k,j})
$ any solution to \eqref{eq:SDE_Z_ter}-\eqref{eq:derivative_Z} for a given $a\in\Rb^{(n+1)K}$ and $k=0,\cdots,n$, $j=1,\cdots,K$.
Throughout the remainder of this section we denote by $C$, indistinctly, any positive constant that does not depend on $a \in\Rb^{(n+1)K}$.

We start with the following stability and moment estimates for the solution to the system \eqref{eq:SDE_Z_ter}-\eqref{eq:derivative_Z}.
\begin{lemma}
\label{lem:estimates_z_Y}
{ Let Assumption \ref{ass:initial} be in force and 
assume that $\alpha(t,\cdot),\beta(t,\cdot)\in C^2(\mathbb{R}^d)$, for any $t\in[0,T]$, with derivatives bounded by some $R>0$}. Then, for any $a\in\Rb^{(n+1)K}$, the system \eqref{eq:SDE_Z_ter}-\eqref{eq:derivative_Z} is strongly well-posed for any $k = 0, \cdots, n$ and $j = 1,\cdots, K$. Furthermore, we have
\begin{align}
\lVert Z^{a}\rVert_{\cS^2_{[0,T]}} + \sum_{\substack{k=0,\cdots,n \\ j=1,\cdots, K}} \lVert Y^{a;k,j} \rVert_{\cS^2_{[0,T]}}   & \leq C , && a\in\Rb^{(n+1)K}  ,\label{eq:estimates_Z_a} \\     \label{eq:estimates_lip_Z_a}
   \lVert Z^{a} - Z^{a'} \rVert_{\cS^2_{[0,T]}} + \sum_{\substack{k=0,\cdots,n \\ j=1,\cdots, K}} \lVert Y^{a;k,j} - Y^{a';k,j} \rVert_{\cS^2_{[0,T]}}     
   &\leq C \,           |a - a' |, && a,a'\in\Rb^{(n+1)K},
    \end{align}
    for some constant $C>0$ independent of $a$. 
\end{lemma}
\begin{proof}
Well-posedness and estimates \eqref{eq:estimates_Z_a}-\eqref{eq:estimates_lip_Z_a} for $Z$ follow directly from Lemma \ref{lem:sde_stability_general} using  that $\h$ is bounded and Lipschitz continuous. In particular, Lemma \ref{lem:sde_stability_general}-(a) yields
\begin{equation}
\lVert Z^{a}\rVert_{\cS^2_{[0,T]}} \leq C, 
\end{equation}
for $C>0$ independent of $a$. 
Lemma \ref{lem:sde_stability_general}-(c) yields 
\begin{align}
\label{eq:estimate_z_a_aprime} \nonumber
  \lVert Z^{a} - Z^{a'} \rVert_{\cS^4_{[0,T]}} 
	& 
	\leq 
	C \|  \h\circ \Lc a - \h\circ \Lc a'  \|_{\infty} 
	\\
	& \leq 
	C \|  \nabla \h  \|_{\infty} \|  \Lc a -  \Lc a'  \|_{\infty} 
	\leq 
	C \|  \nabla \h  \|_{\infty} \Big(\sum_{k=0,\cdots, n} \| g_k \|_{\infty} \Big) |  a -  a'  |,
\end{align}
with $C$ depending on $\| \h \|_{\infty}$ but not on $a$, and $\nabla\h$ stands for the Jacobian of $\h:\Rb^K \to \Rb^K$.

As for $Y$, to ease the notation we only consider the case $q=1$. The general case is completely analogous. We observe that SDE \eqref{eq:derivative_Z} can be cast in the form of \eqref{eq:sde_appendix_krylov} with $\zeta=0$, $m = 2 d(1+d) $, 
\begin{align}
\rho_t 
= \rho_t (a) 
& 
= \Big(  g_k(t) \nabla {\h_j} \big( \Lc a (t) \big) \alpha(t,Z_t), \h\big( \Lc a(t) \big) \partial_{z_1} \alpha(t,   Z_t ) , \cdots,  \h\big( \Lc a(t) \big) \partial_{z_d} \alpha(t,   Z_t ) ,
\\  
&\qquad \quad 
g_k(t) \nabla {\h_j} \big( \Lc a (t) \big) \beta(t,Z_t), \h\big( \Lc a(t) \big) \partial_{z_1} \beta(t,   Z_t ) , \cdots,  \h\big( \Lc a(t) \big) \partial_{z_d} \beta(t,   Z_t ) \Big),
\end{align}
and   
\begin{equation}
\mu(t,y) = \begin{pmatrix}
    I_d  \\
    y_{1} I_d \\
    \vdots\\
    y_{d} I_d \\
    0_{d(1+d),d}
  \end{pmatrix}, \qquad 
  \sigma(t,y) = \begin{pmatrix}
     0_{d(1+d),d} \\
    I_d  \\
    y_{1} I_d \\
    \vdots\\
    y_{d} I_d 
  \end{pmatrix}, \qquad t\in [0,T], \ y=(y_1,\cdots,y_d)\in\mathbb{R}^d,
\end{equation}
where $I_d$ is the $d\times d$ identity matrix and $0_{d(1+d),d}$ is the $d(1+d)\times d$-dimensional matrix with null entries. 
Under our assumptions, the functions $\mu,\sigma$ satisfy conditions \eqref{eq:assump_lip_app}-\eqref{eq:assump_linear_growth} for some $\Lambda>0$ that is independent of $a$ and the stochastic process $\rho$ is bounded by a constant that is also independent of $a$ (recall that by construction $\h$ is constant outside a ball -- see \eqref{eq:h}).
Therefore, by Lemma \ref{lem:sde_stability_general}, \eqref{eq:derivative_Z} is strongly well-posed and we have (by Part (b) of Lemma \ref{lem:sde_stability_general}) that for any $k = 0, \cdots, n$ and $j = 1,\cdots, K$ 
\begin{align}
 \lVert Y^{a;k,j}\rVert_{\cS^2_{[0,T]}}  & \leq C ,  \label{eq:estim_Y_lem} \\     \label{eq:estim_Y_lem_bis}
   \lVert Y^{a;k,j} - Y^{a';k,j} \rVert_{\cS^2_{[0,T]}}     &
   \leq C \,           \|\rho(a) - \rho(a')\|_{\mathbb{L}^4_{[0,T]}}.
    \end{align}
In order to conclude, we write
\begin{align}
| \rho_t(a) - \rho_t(a') |& \leq  |g_k(t)| \Big( \big|  \nabla {\h_j} \big( \Lc a (t) \big) \alpha(t,Z^a_t) -\nabla {\h_j}  \big( \Lc a' (t) \big) \alpha(t,Z^{a'}_t)   \big| \\
&+ \big|  \nabla {\h_j} \big( \Lc a (t) \big) 
\beta(t,Z^a_t) - \nabla {\h_j} \big( \Lc a' (t) \big) \beta(t,Z^{a'}_t)   \big| \Big) \hspace{0pt}\\
&+ \sum_{i=1}^d \Big(  \big|  \h\big( \Lc a(t) \big) \partial_{z_i} \alpha(t,   Z^a_t )  -   \h\big( \Lc a'(t) \big) \partial_{z_i} \alpha(t,   Z^{a'}_t ) \big|\\
&+ \big|  \h\big( \Lc a(t) \big) \partial_{z_i} \beta(t,   Z^a_t )  -   \h\big( \Lc a'(t) \big) \partial_{z_i} \beta(t,   Z^{a'}_t ) \big|   \Big),
\end{align}
whereby from the boundedness of $\h,\alpha,\beta$ and their derivatives, we obtain  
\begin{equation}
| \rho_t(a) - \rho_t(a') | \leq C |\Lc a (t) - \Lc a' (t) | + | Z^a_t - Z^{a'}_t |.
\end{equation}
This then yields, via \eqref{eq:estimate_z_a_aprime}, 
\begin{equation}
  \|\rho(a) - \rho(a')\|^4_{\mathbb{L}^4_{[0,T]}} \leq C |a - a'|^4,
\end{equation}
and completes the proof.
\end{proof}

\begin{lemma}\label{Lemma:4.15.Diff}
{ Let Assumption \ref{ass:initial} be in force. Assume also that $\alpha(t,\cdot),\beta(t,\cdot)\in C^2(\mathbb{R}^d)$, for any $t\in[0,T]$, and $\varphi\in C^{2}(\mathbb{R}^K)$, with derivatives bounded by some $R>0$}. Then, the function $\Rb^{(n+1)K}\ni a\mapsto\Eb \big[    \varphi(Z^{a}_t) \big] $ is differentiable, and 
\begin{equation}
\label{eq:der_expect}
\partial_{a_{k,j}}  \Eb \big[    \varphi(Z^{a}_t) \big]  
= 
\Eb \big[ \nabla_x \varphi \big(Z^{a}_{t})Y^{a;k,j}_{t} \big] , \qquad a\in\Rb^{(n+1)K} , \ t\in[0,T],
\end{equation}
for any $k = 0, \cdots, n$ and $j = 1,\cdots, K$. Furthermore, the functions $\Rb^{(n+1)K}\ni a\mapsto  \Eb \big[ \nabla_x \varphi \big(Z^{a}_{t})Y^{a;k,j}_{t} \big] $ are bounded and Lipschitz continuous, uniformly with respect to $t\in [0,T]$ (for any $k,j$).
\end{lemma}
\begin{proof}
To ease the notation, we prove the statement for $K=1$ and $n=0$, the general case being completely analogous.

Since ${\bf h},\alpha(t,\cdot),\beta(t,\cdot)$ and their derivatives up to order $2$ are bounded, uniformly w.r.t.~$t\in[0,T]$, \cite[Theorem 2.3.1 (p.218)]{kunita1984stochastic} yields 
\begin{equation}
\label{eq:der_Z_as}
\partial_{a}    Z^{a}_t = Y^{a}_{t},\quad a\in\Rb,  \quad \text{a.s.},
\end{equation}
for any $t\in[0,T]$. 
Moreover, given the differentiability of $\varphi$ and the integrability of $Y^a$ from Lemma \ref{lem:estimates_z_Y}, the right-hand side of \eqref{eq:der_expect} is well defined. Therefore, the identity \eqref{eq:der_expect} can be obtained once the exchange of derivative and expectation operations is justified. Let us set
\begin{equation}\label{eq:der_exp_z}
f(t;a):= \Eb \big[ \varphi' \big(Z^{a}_{t})Y^{a}_{t} \big], \qquad t\in [0,T], \ a\in\Rb.
\end{equation}
By Lemma \ref{lem:estimates_z_Y} and the boundedness of $\varphi',\varphi''$ we have
\begin{align}\label{eq:f_bound}
|f(t;a)| &\leq C \Eb \big[|Y^{a}_{t} | \big] \leq C, \\
|f(t;a)- f(t;a')| &\leq \Eb \big[ \big| \varphi' \big(Z^{a}_{t}) - \varphi' \big(Z^{a'}_{t}) \big|\times \big|Y^{a}_{t} \big| \big] + \Eb \big[ \big| \varphi' \big(Z^{a}_{t}) \big| \times \big|Y^{a}_{t}
- Y^{a'}_{t} \big| \big] \leq C |a-a'|,
\end{align}
for any $t\in [0,T]$, $a\in\Rb$, with $C>0$ independent of $t$ and $a$. Therefore, we obtain
\begin{align}
\Eb \big[    \varphi(Z^{a}_t) \big] -  \Eb \big[    \varphi(Z^{0}_t)   \big] 
=   \Eb\big[ \varphi(Z^{a}_t) -    \varphi(Z^{0}_t)   \big]  
& =   \Eb \bigg[   \int_0^a \partial_{\hat a} \big\{\varphi(Z^{\hat a}_t) \big\} \dd \hat a  \bigg]
\\
&= \Eb \bigg[   \int_0^a  \varphi'(Z^{\hat a}_t) Y^{\hat a}_{t}  \dd \hat a  \bigg]
 =       \int_0^a f(t;\hat a)  \dd \hat a,  
\end{align}
where the second last equality follows by using \eqref{eq:der_Z_as}, and the last equality from Fubini's theorem (using \eqref{eq:f_bound}). 

Lastly, using the continuity of $f(t;\cdot)$ we conclude that 
\begin{equation}
\partial_a \Eb \big[    \varphi(Z^{a}_t) \big] = f(t;a), \qquad a\in\Rb , \ t\in[0,T],
\end{equation}
which completes the proof.
\end{proof} 

\subsection{Proof of the main auxiliary propositions}

We are now in the position to prove Proposition \ref{lem:consistency_noise} and \ref{lem:regul_grad_G}. 
\begin{proof}[Proof of Proposition \ref{lem:consistency_noise}]
To ease the notation, we prove the statement for $K=1$ and $n=0$, the general case being completely analogous. 

By \eqref{eq:der_expect} in Lemma \ref{Lemma:4.15.Diff} we have
\begin{equation}
 \partial_{a} \Big(  \big| \Eb \big[ \varphi(Z^{a}_t) -  \h\big((\Lc a) (t)\big) \big] \big|^2   \Big) = 2\, \Eb \big[ \varphi(Z^{a}_t) -  \h\big((\Lc a) (t)\big) \big] \Eb \big[ \varphi' \big(Z^{a}_{t}) Y^{a}_{t} -  \partial_a \h\big((\Lc a) (t)\big) \big].
\end{equation}
Recall the heuristics in \eqref{eq:formal_comp} and \eqref{eq:v_function}. For any independent copies $(\xi,W)$ and $(\tilde\xi,\tilde W)$, since $\h,\h',\varphi,\varphi'$ are bounded, we observe that 
\begin{align}
\label{eq:bound_der_abs}
\nonumber
\Eb\big[ \big| \varphi \big(Z^{a}_{t}(\bar\xi,\bar W)\big) - \h\big((\Lc a) (t)\big) \big|^2      \big] 
\times  
\Eb\big[ \big|  \varphi' \big(Z^{a}_{t}(\tilde\xi,\tilde W)\big) Y^{a}_{t}(\tilde\xi,\tilde W)\big)   - 
&
\partial_{a} \h\big( (\Lc a) (t) \big)  \big|^2      \big]
\\
&  
\leq C \big(  1   +    \lVert Y^{a} \rVert_{\cS^2_{[0,T]}}  \big)\leq C,
\end{align}
where the last inequality follows from \eqref{eq:estimates_Z_a}. Therefore, we can exchange the integration and derivative operation to obtain the following
\begin{equation}
 \partial_a G(a)  =  
 \Eb[        v(a;  \xi, {W} ; \tilde\xi, \tilde W)   ] + \partial_a H(a),
 \end{equation}
with $v$ as defined in \eqref{eq:v_function}, namely
\begin{equation}\label{eq:v_function_bis}
v(a;  \xi, {W} ; \tilde\xi, \tilde W)   : = 2 \int_0^T       \Big(  \varphi \big(Z^{a}_{t}(\xi, W)\big) - \h\big((\Lc a) (t)\big) \Big) \Big(  \varphi' \big(Z^{a}_{t}(\tilde\xi,\tilde W)\big) Y^{a}_{t}(\tilde\xi,\tilde W)   - \partial_{a} \h\big( (\Lc a) (t) \big)       \Big)               \dd t.
\end{equation}
Moreover, Jensen's inequality and Fubini's theorem, together with \eqref{eq:bound_der_abs}, prove that 
\begin{equation}\label{eq:estimate_v_second}
\Eb_{\bar{\Pb}}\big[ | v( a ; \xi, W;\tilde\xi, \tilde W) |^2  \big] \leq C, \qquad a\in\Rb, \quad m\in \mathbb{N}_0,
\end{equation}
with $C$ independent of $a$. 
Hence, for $m\in\mathbb{N}$ and recalling ${\bf v}_{m}$ as defined recursively in Algorithm \ref{alg:SGDMVSDE}, we obtain for any $m \in \mathbb{N}_0$ 
\begin{align}
&\hspace{-8pt}\Eb_{\bar{\Pb}}\big[ {\bf v}_{m+1} - \nabla_a G( {\bf a}_m) | \bar{\mathcal{F}}_m \big]\\
&\hspace{-8pt} = \Eb_{\bar{\Pb}}\Big[ v({\bf a}_m; \xi_{m+1}, W_{m+1};\tilde\xi_{m+1}, \tilde W_{m+1}) 
- \Eb_{\bar{\Pb}}\big[  v(a; \xi_{m+1}, W_{m+1};\tilde\xi_{m+1}, \tilde W_{m+1})  \big]_{a={\bf a}_m}\Big| \bar{\mathcal{F}}_m \Big] = 0, 
\end{align}
which is \eqref{eq:consistency_condition}. From \eqref{eq:estimate_v_second} and Assumption \ref{ass:H} on $H$ we also obtain
\begin{equation}
\Eb_{\bar{\Pb}}\big[ | {\bf v}_{m+1} - \nabla_a G( {\bf a}_m)  |^2 \big| \bar{\mathcal{F}}_m \big] \leq C, \qquad m \in \mathbb{N}_0,
\end{equation}
which concludes the proof.
\end{proof}

\begin{proof}[Proof of Proposition \ref{lem:regul_grad_G}]
To ease the notation, we prove the statement for $K=1$ and $n=0$, the general case being completely analogous.

By the same arguments as in the proof of Proposition \ref{lem:consistency_noise}, and using Assumption \ref{ass:H}, one obtains
\begin{equation}
\partial_a G(a) = 2 \int_0^T   \Eb\Big[      \varphi \big(Z^{a}_{t}\big) - \h\big((\Lc a) (t)\big) \Big] \Eb\Big[  \varphi' \big(Z^{a}_{t}\big) Y^{a}_{t}   - \partial_{a} \h\big( (\Lc a) (t) \big)       \Big]               \dd t  + \partial_a H(a).
\end{equation}
From this expression we obtain the identity 
\begin{equation}
\partial_a G(a) -  \partial_a G(a')  = 2 \int_0^T \big( I_1(t) + I_2(t) \big) \dd t + I_3, 
\end{equation}
where
\begin{align}\label{eq:decomp_grad_G}
 I_1(t) & =  \Eb\Big[      \varphi \big(Z^{a}_{t}\big) - \varphi \big(Z^{a'}_{t}\big) +  \h\big((\Lc a') (t)\big) - \h\big((\Lc a) (t)\big) \Big] \Eb\Big[  \varphi' \big(Z^{a}_{t}\big) Y^{a}_{t}   - \partial_{a} \h\big( (\Lc a) (t) \big)       \Big]  , \\
 I_2(t) & =   \Eb\Big[      \varphi \big(Z^{a'}_{t}\big) - \h\big((\Lc a') (t)\big) \Big]  \Eb\Big[  \varphi' \big(Z^{a}_{t}\big) Y^{a}_{t} - \varphi' \big(Z^{a'}_{t}\big) Y^{a'}_{t} + \partial_{a} \h\big( (\Lc a') (t) \big)   - \partial_{a} \h\big( (\Lc a) (t) \big)       \Big]  , \\
 I_3 & =    \partial_a H(a) - \partial_a H(a')   .   
\end{align}
Now, by the boundedness of $\varphi,\varphi',\varphi'',{\bf h}, {\bf h}'$ we have
\begin{align}\label{eq:estim_decomp}
| I_1 (t) | & \leq C\,  \Eb\big[      | Z^{a}_{t} - Z^{a'}_{t}| + |a - a'| \big] \Eb\big[      | Y^{a}_{t}  | + 1 \big] , \\
| I_2 (t) | & \leq C\,  \Eb\big[  |a - a'| (   | Y^{a}_{t}| + 1 ) +  | Y^{a}_{t} - Y^{a'}_{t}|  \big] . 
\end{align}
Therefore, Lemma \ref{lem:estimates_z_Y} and Assumption \ref{ass:H} yield the following result
\begin{equation}\label{eq:grad_G_lip_final}
| \partial_a G(a) -  \partial_a G(a')  | \leq C |a - a'|, \qquad a\in\Rb.
\end{equation}  
\end{proof}

\end{document}